\documentclass[a4paper,10pt]{article}
\usepackage[latin1]{inputenc}
\usepackage[T1]{fontenc}
\usepackage{amsmath}
\usepackage{amsfonts}
\usepackage{amstext}
\usepackage{amsthm}
\usepackage[all]{xy}
\usepackage{geometry}
\usepackage{graphics}
\usepackage{epsfig}
\usepackage{subfigure}
\usepackage{slashed}
\usepackage{color}
\usepackage{amssymb}
\usepackage{srcltx}
\newtheorem{theorem}{Theorem}[section]
\newtheorem{lemma}[theorem]{Lemma}
\newtheorem{defi}[theorem]{Definition}
\newtheorem{rem}[theorem]{Remark}
\newtheorem{prop}[theorem]{Proposition}

\DeclareMathOperator{\supp}{supp}
\DeclareMathOperator{\im}{im}
\DeclareMathOperator{\ind}{ind}
\DeclareMathOperator{\sind}{s-ind}

\DeclareMathOperator{\sfl}{sf}

\DeclareMathOperator{\GL}{GL}

\DeclareMathOperator{\Sp}{Sp}
\DeclareMathOperator{\gra}{graph}

\title{A family index theorem for periodic Hamiltonian systems and bifurcation}
\author{Nils Waterstraat}

\begin{document}
\date{}
\maketitle

\footnotetext[1]{{\bf 2010 Mathematics Subject Classification: Primary 58J20; Secondary 37J20,37J45,19K56 }}
\footnotetext[2]{Keywords: Hamiltonian systems, index bundle, spectral flow, multiparameter bifurcation}
\footnotetext[3]{The author was supported by a postdoctoral fellowship of
the German Academic Exchange Service (DAAD) and by the program "Professori Visitatori Junior" of GNAMPA-INdAM.}

\begin{abstract}
We prove an index theorem for families of linear periodic Hamiltonian systems, which is reminiscent of the Atiyah-Singer index theorem for selfadjoint elliptic operators. For the special case of one-parameter families, we compare our theorem with a classical result of Salamon and Zehnder. Finally, we use the index theorem to study bifurcation of branches of periodic solutions for families of nonlinear Hamiltonian systems.
\end{abstract}

\section{Introduction}
Let $X$ be a compact topological space and $Y\subset X$ a closed subspace. We denote by

\begin{align}\label{standardmatrix}
J=\begin{pmatrix}
0&-I_n\\
I_n&0
\end{pmatrix}
\end{align}
the standard symplectic matrix of dimension $2n$ and set $S^1=\mathbb{R}/2\pi\mathbb{Z}$.
We consider Hamiltonian systems

\begin{equation}\label{equation}
\left\{
\begin{aligned}
Ju'(t)+\nabla_u&\mathcal{H}(\lambda,t,u(t))=0,\quad t\in [0,2\pi]\\
u(0)&=u(2\pi),
\end{aligned}
\right.
\end{equation}
where $\mathcal{H}:X\times \mathbb{R}\times\mathbb{R}^{2n}\rightarrow\mathbb{R}$ is a continuous function such that $\mathcal{H}_\lambda:\mathbb{R}\times\mathbb{R}^{2n}\rightarrow\mathbb{R}$, $\lambda\in X$, is $C^2$ and all its partial derivatives depend continuously on the parameter $\lambda\in X$. Moreover, we assume that $\mathcal{H}(\lambda,t,u)$ is $2\pi$-periodic in $t$, $\nabla_u\mathcal{H}(\lambda,t,0)=0$ for all $(\lambda,t)\in X\times\mathbb{R}$ and that there exist constants $a,b\geq 0$ and $r>1$ such that

\begin{align}\label{growthcond}
\begin{split}
|\nabla_u\mathcal{H}(\lambda,t,u)|&\leq a+b|u|^r,\\
|D_u\nabla_u\mathcal{H}(\lambda,t,u)|&\leq a+b|u|^r,\quad (\lambda,t,u)\in X\times\mathbb{R}\times\mathbb{R}^{2n}.
\end{split}
\end{align}
Note that in particular $u\equiv 0$ is a solution of all systems \eqref{equation}.\\
Under the assumption \eqref{growthcond} the equations \eqref{equation} have a variational formulation on the Hilbert space $H^\frac{1}{2}(S^1,\mathbb{R}^{2n})$, that is, there exists a family $\psi:X\times H^\frac{1}{2}(S^1,\mathbb{R}^{2n}) \rightarrow\mathbb{R}$ of functionals, such that $u\in H^\frac{1}{2}(S^1,\mathbb{R}^{2n})$ is a (weak) solution of \eqref{equation} for the parameter value $\lambda\in X$ if and only if it is a critical point of $\psi_\lambda:=\psi(\lambda,\cdot)$. Let us denote by $L=\{L_\lambda\}_{\lambda\in X}$, the family of Hessians of $\psi$ at the branch of trivial critical points $u\equiv 0$, which is a family of bounded selfadjoint Fredholm operators on $H^\frac{1}{2}(S^1,\mathbb{R}^{2n})$. Note that these operators play a fundamental role in the study of existence and multiplicity of solutions of the nonlinear equations \eqref{equation} (cf. eg. \cite[Chapter IV]{Chang}). We use a relative version of the classical index bundle from \cite{AtiyahSinger} to assign to the family $L$ a $K$-theory class in $K^{-1}(X,Y)$, which we denote by $\mu_{Morse}(\mathcal{H})$ and call the \textit{generalised Morse index} of \eqref{equation}.\\
Now consider the family of linear Hamiltonian systems

\begin{align}\label{diffop}
Ju'+S_\lambda(\cdot)u=0,\quad\lambda\in X,
\end{align}
where

\begin{align}\label{Smatrix}
S_\lambda(t):=D_u\nabla_u\mathcal{H}(\lambda,t,0),\quad(\lambda,t)\in X\times\mathbb{R},
\end{align} 
is the Hessian of $\mathcal{H}(\lambda,t,\cdot)$ at $0\in\mathbb{R}^{2n}$. We regard the left hand side of \eqref{diffop} as unbounded selfadjoint Fredholm operators $\mathcal{A}_\lambda$, $\lambda\in X$, on $L^2(S^1,\mathbb{R}^{2n})$ having domains $H^{1}(S^1,\mathbb{R}^{2n})$. By using again our variant of the index bundle as for the generalised Morse index $\mu_{Morse}(\mathcal{H})$, we assign to the family $\mathcal{A}=\{\mathcal{A}_\lambda\}_{\lambda\in X}$ a relative $K$-theory class $\mu_{spec}(\mathcal{H})\in K^{-1}(X,Y)$ called the \textit{spectral index} of \eqref{equation}. \\
Let us recall that for a family $D=\{D_\lambda\}_{\lambda\in X}$, of selfadjoint elliptic operators acting on a Hermitian bundle over a closed manifold, the classical index bundle \cite{AtiyahSinger} defines the \textit{analytical index} in the Atiyah-Singer theorem \cite{AtiyahPatodi}. The analytical index belongs to the odd $K$-theory group $K^{-1}(X)$ and according to the Atiyah-Singer theorem it is given by the \textit{topological index} of the family $D$, which only depends on the family of principal symbols of $D$ and topological invariants of the Hermitian bundle. Note that our operators $\mathcal{A}_\lambda$, $\lambda\in X$, in the definition of $\mu_{spec}(\mathcal{H})$ are in fact selfadjoint elliptic differential operators acting on a product bundle over the manifold $S^1$. However, the analytical index of the family $\mathcal{A}$ is easily seen to vanish by deforming $\mathcal{A}$ to a family of invertible operators. As we shall see later, we can avoid triviality of $\mu_{spec}(\mathcal{H})$ in general because we consider relative $K$-theory classes in $K^{-1}(X,Y)$.\\
Finally, the family of monodromy matrices of the differential equations \eqref{diffop} induce canonically an element of $K^{-1}(X,Y)$ which we denote by $\mu_{mon}(\mathcal{H})$ and call the \textit{monodromy index}. Our main theorem states the equality of $\mu_{Morse}(\mathcal{H})$, $\mu_{spec}(\mathcal{H})$ and $\mu_{mon}(\mathcal{H})$ and so computes the rather involved objects $\mu_{Morse}(\mathcal{H})$ and $\mu_{spec}(\mathcal{H})$ in terms of a simple matrix family.\\  
Let us now describe the content and the structure of the paper. The second section is devoted to the construction of the relative version of the index bundle for families of generally unbounded selfadjoint Fredholm operators having a fixed domain. In particular, we discuss its relation to the spectral flow, and we state an abstract index theorem holding at the operator theoretic level which will later be used for showing the equality of $\mu_{Morse}(\mathcal{H})$ and $\mu_{spec}(\mathcal{H})$ in our main index theorem for Hamiltonian systems. At this point we recall briefly the Morse index theorem for geodesics in semi-Riemannian manifolds \cite{Pejsachowicz} and explain that our abstract index theorem immediately solves a problem that was left open in the topological proof of the Morse index theorem in the authors article \cite{MorseIch}. The reader may find in Appendix A the basics of $K$-theory that are needed to follow our arguments in the second section and the rest of the paper. In the third section we introduce our indices $\mu_{Morse}(\mathcal{H})$, $\mu_{spec}(\mathcal{H})$ and $\mu_{mon}(\mathcal{H})$ and we state the index theorem for Hamiltonian systems. A particular special case of our theorem appears if the parameter space $X$ is a compact interval and $Y$ is its boundary, which we discuss in the fourth section. Then the operator families $L$ and $\mathcal{A}$ are paths of selfadjoint Fredholm operators and moreover there exists an isomorphism $c_1:K^{-1}(X,Y)\rightarrow\mathbb{Z}$ induced by the first Chern number. Under this isomorphism the generalised Morse index $\mu_{Morse}(\mathcal{H})$ and the spectral index $\mu_{spec}(\mathcal{H})$ correspond to the spectral flows of the paths $L$ and $\mathcal{A}$ respectively, while the monodromy index $\mu_{mon}(\mathcal{H})$ becomes the winding number of a certain planar vector field. Let us mention that the spectral flows of $\mathcal{A}$ and $L$ were identified with the Conley-Zehnder index of the path of monodromy matrices of \eqref{diffop} by Salamon and Zehnder in \cite{Salamon-Zehnder} and Fitzpatrick, Pejsachowicz, Recht in \cite{SFLPejsachowiczII}, respectively. The integer that we obtain from $\mu_{mon}(\mathcal{H})$ in this case seems to us rather different in nature from the Conley-Zehnder index. As an application, we use well known methods for computing winding numbers of planar vector fields in order to calculate our indices explicitly for analytic Hamiltonian systems under additional assumptions. The reader may find a brief survey on the spectral flow in Appendix B, which, however, is not intended to be exhaustive and rather adapted to our setting. In the fifth section we consider bifurcation of branches of periodic solutions of the nonlinear Hamiltonian systems \eqref{equation}. We improve the main result of \cite{SFLPejsachowiczII} on paths of Hamiltonian systems to families by using the approach of \cite{JacBifIch} and our main index theorem. The final sixth section is devoted to the proofs of the abstract index theorem for the index bundle from the second section, the index theorem for Hamiltonian systems from the third section and the bifurcation theorem stated in the fifth section.


\section{On the index bundle}\label{section:indbundlewhole}

We introduce in this section a way to construct the index bundle for families of (generally unbounded) selfadjoint Fredholm operators having a fixed domain. In Section \ref{section:indbundle} we recall at first the definition of the index bundle for Fredholm morphisms between Banach bundles from \cite{indbundleIch} under the additional assumption that the target bundle is a product. The following Section \ref{section:sindbundle} is divided into two parts. In Section \ref{section:prop} we define the index bundle for selfadjoint Fredholm operators and discuss its relation to the spectral flow as defined by Robbin and Salamon in \cite{Robbin-Salamon}. Afterwards, in Section \ref{section:reduction} we state an abstract index theorem which allows to assign to a family of unbounded selfadjoint Fredholm operators a family of bounded selfadjoint Fredholm operators having the same index bundle. The abstract index theorem will play an essential role in the proof of our index theorem for Hamiltonian systems. 


\subsection{The index bundle for Fredholm morphisms}\label{section:indbundle}
The index bundle of a continuous family of bounded Fredholm operators parametrised by a compact space is an element of the $K$-theory group of the parameter space with the same properties as the ordinary integral-valued index of a single Fredholm operator. It was introduced in \cite{AtiyahSingerFam} in the proof of the family-version of the famous Atiyah-Singer theorem and independently by Jänich in \cite{Jaenich}. Here we recall its definition in the more general case of Fredholm morphisms between Banach bundles, where we make throughout the simplifying assumption that the target bundle is a product. The more involved construction for Fredholm morphisms mapping into general Banach bundles can be found in \cite{indbundleIch}. However, in contrast to the aforementioned references, we allow locally compact parameter spaces and we define the index bundle as a relative $K$-theory class.\\
Let $X$ be a locally compact topological space and $Y\subset X$ a closed subspace. Let $\mathcal{E}$ be Banach bundle over $X$ and $F$ a Banach space. In what follows, $\Theta(F)$ stands for the product bundle over $X$ with fibre $F$. We denote by $\mathcal{F}_{0,c}(\mathcal{E},\Theta(F))$ the set of all bundle morphisms $T:\mathcal{E}\rightarrow\Theta(F)$ such that

\begin{itemize}
	\item $T_\lambda:\mathcal{E}_\lambda\rightarrow F$ is a Fredholm operator of index $0$, $\lambda\in X$, 
	\item there exists a compact subset $K\subset X$ such that $T_\lambda$ is invertible for all $\lambda\in X\setminus K$.
\end{itemize}
We include a proof of the following well known result for later reference in Section \ref{proof:reduction}.

\begin{lemma}\label{transversal}
For any $T\in\mathcal{F}_{0,c}(\mathcal{E},\Theta(F))$ there exists a finite dimensional subspace $V\subset F$ such that

\begin{align}\label{transversality}
\im(T_\lambda)+V=F,\quad \lambda\in X.
\end{align} 
\end{lemma}

\begin{proof}
Let $\lambda_0\in X$ and $U_{\lambda_0}\subset X$ be an open neighbourhood of $\lambda_0$ such that $\mathcal{E}$ is trivial on $U_{\lambda_0}$ by a trivialisation $\varphi$. Let $\tilde{T}=T\circ\varphi^{-1}:U_{\lambda_0}\times E\rightarrow F$ be the corresponding family of bounded Fredholm operators with respect to this trivialisation, where $E$ denotes the model space of $\mathcal{E}$. Since $\tilde{T}_{\lambda_0}$ is Fredholm, there exists $V_{\lambda_0}\subset F$, $\dim V_{\lambda_0}<\infty$, and $W_{\lambda_0}\subset E$ closed such that
 
\[\im(\tilde T_{\lambda_0})\oplus V_{\lambda_0}=F, \qquad \ker(\tilde T_{\lambda_0})\oplus W_{\lambda_0}=E.\]
Now consider

\[A_{\lambda}:W_{\lambda_0}\times V_{\lambda_0}\rightarrow F,\;\; A_{\lambda}(w,v)=\tilde T_\lambda w+v\]
and recall that the set of invertible operators $\GL(W_{\lambda_0}\times V_{\lambda_0},F)$ is open in the space of bounded operators $\mathcal{L}(W_{\lambda_0}\times V_{\lambda_0},F)$ with respect to the norm topology. Because of $A_{\lambda_0}\in \GL(W_{\lambda_0}\times V_{\lambda_0},F)$ and the continuity of $A:U_{\lambda_0}\rightarrow \mathcal{L}(W_{\lambda_0}\times V_{\lambda_0},F)$, there exists a neighbourhood $\tilde U_{\lambda_0}\subset U_{\lambda_0}$ of $\lambda_0$ such that $A_\lambda\in \GL(W_{\lambda_0}\times V_{\lambda_0},F)$ for all $\lambda\in\tilde U_{\lambda_0}$ and hence

\[\im(T_\lambda)+V_{\lambda_0}=\im(\tilde T_\lambda)+V_{\lambda_0}=F\;\;\text{ for all }\; \lambda\in \tilde U_{\lambda_0}.\]
Let now $K\subset X$ be a compact subset such that $T_\lambda$ is invertible for all $\lambda\in X\setminus K$. We cover $K$ by a finite number of neighbourhoods $\tilde{U}_{\lambda_i}$, $i=1,\ldots,n$, such that for each $i$ there exists a finite dimensional subspace $V_{\lambda_i}$ such that

\[\im(T_\lambda)+V_{\lambda_i}=F\;\;\text{for all}\,\, \lambda\in\tilde{U}_{\lambda_i},\quad i=1,\ldots,n.\] 
Then $V:=V_1+\ldots+V_n$ defines a finite dimensional subspace of $F$ satisfying \eqref{transversality} on all of $X$. 
\end{proof}

Let $T\in\mathcal{F}_{0,c}(\mathcal{E},\Theta(F))$ be a Fredholm morphism and $V\subset F$ a finite dimensional subspace as in \eqref{transversality}. We obtain a surjective bundle morphism

\[\mathcal{E}\xrightarrow{T}\Theta(F)\rightarrow\Theta(F/V)\]
and according to \cite[III,\S 3]{Lang} the kernel of this morphism defines a finite dimensional subbundle $E(T,V)$ of $\mathcal{E}$ having the fibres

\[\{u\in\mathcal{E}_\lambda:\,T_\lambda u\in V\},\quad\lambda\in X.\]
Note that

\begin{align}\label{dimE}
\dim E(T,V)=\dim V
\end{align}
and $T$ restricts to a bundle morphism $T\mid_{E(T,V)}:E(T,V)\rightarrow\Theta(V)$. Moreover, $T_\lambda:\mathcal{E}_\lambda\rightarrow F$ is an isomorphism if and only if $T\mid_{E(T,V)_\lambda}$ is an isomorphism, $\lambda\in X$. 

\begin{defi}
Let $T\in\mathcal{F}_{c,0}(\mathcal{E},\Theta(F))$ be a Fredholm morphism and $Y\subset X$ closed such that $T_\lambda$ is invertible for all $\lambda\in Y$. The index bundle of $T$ is defined by

\begin{align*}
\ind(T)=[E(T,V),\Theta(V),T\mid_{E(T,V)}]\in K(X,Y),
\end{align*}
where $V\subset F$ is any finite dimensional subspace as in \eqref{transversality}.
\end{defi} 

The reader may check that the definition does not depend on the choice of the subspace $V\subset F$ and that the following fundamental properties hold:

\begin{itemize}
	\item[i)] Let $T\in\mathcal{F}_{0,c}(\mathcal{E},\Theta(F))$ be an isomorphism, i.e. $T_\lambda$ is invertible for all $\lambda\in X$. Then \[\ind(T)=0\in K(X,Y).\]
	\item[ii)] Let $\pi:X\times[0,1]\rightarrow X$ be the projection onto the first component. If $h\in\mathcal{F}_{0,c}(\pi^\ast\mathcal{E},\Theta(F))$ is a morphism such that $h_{(\lambda,t)}$ is invertible for all $(\lambda,t)\in Y\times[0,1]$, then 
\[\ind(h\mid_{X\times\{0\}})=\ind(h\mid_{X\times\{1\}})\in K(X,Y).\]
	 
	\item[iii)] Let $G$ be a Banach space, $T\in\mathcal{F}(\mathcal{E},\Theta(F))$ and $S\in\mathcal{F}_{0,c}(\Theta(F),\Theta(G))$ such that $T_\lambda,S_\lambda$ are invertible for all $\lambda\in Y$. Then
\[\ind(S\circ T)=\ind(S)+\ind(T)\in K(X,Y).\]
\end{itemize}


\subsection{The index bundle for selfadjoint families}\label{section:sindbundle}
The index bundle for families of selfadjoint Fredholm operators was introduced by Atiyah, Patodi and Singer in the proof of the index theorem for selfadjoint elliptic operators in \cite{AtiyahPatodi} and it is connected to the eta-invariant and the spectral flow (cf. also Proposition \ref{prop:sfl} below). We will give here a slightly different definition which is particularly adapted to families of operators of the type considered by Robbin and Salamon in \cite{Robbin-Salamon}.

\subsubsection{Definition, properties and spectral flow}\label{section:prop}
Let $W$ and $H$ be infinite dimensional complex Hilbert spaces with a dense injection $W\hookrightarrow H$. Let $\mathcal{L}(W,H)$ be the usual Banach space of all bounded operators and let $\mathcal{K}(W,H)$ be the subspace of $\mathcal{L}(W,H)$ consisting of all compact operators. We denote by $\mathcal{S}(W,H)\subset\mathcal{L}(W,H)$ the elements which are selfadjoint when considered as operators in $H$ with dense domain $W$. Finally, we let $\mathcal{FS}(W,H)\subset\mathcal{S}(W,H)$ consist of all selfadjoint Fredholm operators and we write $\mathcal{GS}(W,H)\subset\mathcal{FS}(W,H)$ for the subspace of invertible elements. In the case $W=H$ we will shorten notation as usual and write e.g. $\mathcal{L}(H,H)=\mathcal{L}(H)$. Note that $A\in\mathcal{S}(W,H)$ belongs to $\mathcal{FS}(W,H)$ if and only if $A$ has a closed range and a finite dimensional kernel. Moreover, $A\in\mathcal{FS}(W,H)$ is an element of $\mathcal{GS}(W,H)$ if and only if $\ker A=0$.\\
Now let $X$ be a compact topological space and $Y\subset X$ a closed subspace. We assume that $D:X\rightarrow\mathcal{FS}(W,H)$ is a continuous family such that $D_\lambda$ is invertible for all $\lambda\in Y$. If we regard $D$ as a bundle morphism between the product bundles $\Theta(W)$ and $\Theta(H)$ over $X$, then $D$ is Fredholm of index $0$ in each fibre and we may assign the index bundle $\ind(D)\in K(X,Y)$ as defined in the previous section. However, since the operators $D_\lambda$, $\lambda\in X$, are selfadjoint when considered as operators in $H$ having dense domain $W$, their spectra are real. Consequently, the homotopy 

\begin{align*}
h:[0,1]\times(X\times W)\rightarrow H,\quad h_{(t,\lambda)}u=D_\lambda u+it\,u
\end{align*}
deforms $D$ to a family of invertible operators such that $h_{(t,\lambda)}$ is invertible for all $(t,\lambda)\in[0,1]\times Y$. We conclude $\ind(D)=\ind(h_1)=0\in K(X,Y)$ by the properties i) and ii) of Section \ref{section:indbundle}.\\
Instead, we now define a new family of bounded operators in $\mathcal{L}(W,H)$ by the following suspension of the family $D$:

\begin{align*}
\overline{D}_zu=D_\lambda u+is\,u,\quad u\in W,\,\, z=(\lambda,s)\in X\times\mathbb{R}.
\end{align*} 
From the selfadjointness of $D_\lambda$, $\lambda\in X$, it follows that $\overline{D}_z$, $z=(\lambda,s)\in X\times\mathbb{R}$, is invertible if $s\neq 0$. For that reason, $\overline{D}$ can be regarded as a Fredholm morphism in $\mathcal{F}_{0,c}(\Theta(W),\Theta(H))$, where now $\Theta(W)$ and $\Theta(H)$ are product bundles over $X\times\mathbb{R}$. Moreover, since we assume that $D_\lambda$ is invertible for all $\lambda\in Y$, $\overline{D}_z$ is invertible whenever $z=(\lambda,s)$ belongs to $Y\times\mathbb{R}$.

\begin{defi}
Let $X$ be a compact topological space, $Y\subset X$ a closed subspace and $D:X\rightarrow\mathcal{FS}(W,H)$ a continuous family of selfadjoint Fredholm operators such that $D_\lambda$ is invertible for all $\lambda\in Y$. The index bundle of $D$ is defined by

\[\sind(D)=\ind(\overline{D})\in K^{-1}(X,Y).\]
\end{defi} 

The following properties of the index bundle for selfadjoint Fredholm operators can easily be derived from the assertions i)-iii) in Section \ref{section:indbundle}:

\begin{enumerate}
	\item[i)] If $D:X\rightarrow\mathcal{FS}(W,H)$ is such that $D_\lambda$ is invertible for all $\lambda\in X$, then \[\sind(D)=0\in K^{-1}(X,Y).\]
	\item[ii)] Let $X_0$ be a compact space, $Y_0\subset X_0$ closed and $f:(X_0,Y_0)\rightarrow(X,Y)$ continuous. Then
	
	\[	\sind(f^\ast\,D)=f^\ast\sind(D),\] 
	where $(f^\ast\,D)_\lambda=D_{f(\lambda)}$, $\lambda\in X_0$.
	\item[iii)] If $D_1,D_2:X\rightarrow\mathcal{FS}(W,H)$ are two families which are invertible on $Y$, then
	\[\sind(D_1\oplus D_2)=\sind(D_1)+\sind(D_2)\in K^{-1}(X,Y).\]
	\item[iv)] Let $h:[0,1]\times X\rightarrow\mathcal{FS}(W,H)$ be a homotopy such that $h(t,\lambda)$ is invertible for all $(t,\lambda)\in[0,1]\times Y$. Then
	
	\[\sind(h_0)=\sind(h_1)\in K^{-1}(X,Y).\]
	\item[v)] Let $K:X\rightarrow\mathcal{S}(H)$ be a family of selfadjoint operators such that $K_\lambda\mid_W\in\mathcal{K}(W,H)$, $\lambda\in X$. If $D_\lambda+t\,K_\lambda$ is invertible for all $\lambda\in Y$ and $t\in [0,1]$, then
	
	\[\sind(D+K)=\sind(D)\in K^{-1}(X,Y).\]
\end{enumerate}

\begin{rem}\label{rem:trivial}
The property v) implies in particular that $\sind(D+K)=\sind(D)$ if $Y=\emptyset$. Below we consider families of compact perturbations of fixed operators and hence a non-empty subspace $Y\subset X$ is essential for obtaining non-trivial index bundles.
\end{rem}

The \textit{spectral flow} is an integer-valued homotopy invariant of paths of selfadjoint Fredholm operators which computes the crossing through $0$ of the spectrum of the operators along the path (cf. Appendix B). It was originally introduced by Atiyah, Patodi and Singer in \cite{AtiyahPatodi} for closed paths, i.e. $X=S^1$, $Y=\emptyset$, and in this case it can be computed as the first Chern number of the analytical index as an element of $K^{-1}(S^1)$ (cf. \cite[\S 7]{AtiyahPatodi}). In view of Remark \ref{rem:trivial}, our aim is to obtain a similar result for generally non-closed paths in $\mathcal{FS}(W,H)$ having invertible endpoints, i.e. $X=I$, $Y=\partial I$, where $I$ denotes any compact interval in $\mathbb{R}$. First recall from \eqref{Chern} the definition of the isomorphism $c_1:K^{-1}(I,\partial I)\rightarrow\mathbb{Z}$.

\begin{prop}\label{prop:sfl}
Let $D:(I,\partial I)\rightarrow(\mathcal{FS}(W,H),\mathcal{GS}(W,H))$ be a path having invertible endpoints. Denote by $\sfl(D)$ its spectral flow. Then

\[\sfl(D)=c_1(\sind(D))\in\mathbb{Z}.\]
\end{prop}

\begin{proof}
We set $\mu(D):=c_1(\sind(D))\in\mathbb{Z}$ and show that $\mu(D)=\sfl(D)$ by using Theorem \ref{Lesch}. Note that $\mu$ satisfies the first three assumptions in Theorem \ref{Lesch} by i),iii) and iv) above.\\
Let now $P$ be a rank one orthogonal projection such that $\ker(P)\subset W$ and $T$ an operator as in iv) of Theorem \ref{Lesch}. We consider the path 
\[D_\lambda=\lambda P+(I_H-P)T(I_H-P), \quad\lambda\in[-1,1],\] of selfadjoint Fredholm operators and we have to show that $\mu(D)=1$. $D_\lambda$ is not invertible only if $\lambda=0$ having kernel $V:=\im(P)$. Since $H=\ker( D_0)\oplus\im(D_0)$, we conclude that $V$ is transversal to the image of the family of operators $\overline{D}_z=D_\lambda+is\,I_H$, $z=(\lambda,s)\in [-1,1]\times\mathbb{R}$, in the sense of  \eqref{transversality}. Moreover, $V$ and $V^\perp=\im(I_H-P)$ are both invariant subspaces under the operators $\overline{D}_z$. Denoting by $\overline{D}\mid_V$ the restriction of the family $\overline{D}$ to $V$, it follows from the definition of the index bundle that

\[\sind(D)=\ind(\overline{D})=[\Theta(V),\Theta(V),\overline{D}\mid_V]\in K^{-1}(X,Y).\]
Since $\dim V=1$, we can identify $V$ with $\mathbb{C}$ and obtain

\[\sind(D)=[\Theta(\mathbb{C}),\Theta(\mathbb{C}),a]\in K^{-1}(I,\partial I),\] 
where $a(z)=z$ for $z=\lambda+is\in[-1,1]\times i\,\mathbb{R}$. Applying \eqref{Chern} yields $c_1(\sind(D))=1$.
\end{proof}


\subsubsection{An abstract index theorem}\label{section:reduction}

Let $W,E$ and $H$ be infinite dimensional complex Hilbert spaces such that $W\subset E\subset H$ as sets. We assume that the inclusion $W\hookrightarrow E$ is continuous, $E\hookrightarrow H$ is compact and that $W$ is dense in $H$. We define an injective bounded linear operator $K:H\rightarrow E$ by requiring that

\begin{align}\label{scalarK}
\langle u,v\rangle_H=\langle Ku,v\rangle_E,\quad u\in H,\,v\in E,
\end{align}
and let $K_0$ be the restriction of $K$ to $E$. Then $K_0$ is the Riesz representation of the scalar product of $H$ as a bounded bilinear form on $E$, i.e, 

\begin{align}\label{scalarK0}
\langle u,v\rangle_H=\langle K_0u,v\rangle_E,\quad u,v\in E,
\end{align}
and it is a nonnegative compact operator. In what follows, we require that the spaces $W$, $E$ and $H$ satisfy the compatibility condition 

\begin{align}\label{red:comp}
W\subset\im(K).
\end{align}
As in the previous section, let $X$ be a compact space, $Y\subset X$ a closed subspace and  $D:X\rightarrow\mathcal{FS}(W,H)$ a continuous family of selfadjoint Fredholm operators such that $D_\lambda$ is invertible for all $\lambda\in Y$. We suppose that there exists $C>0$ such that

\begin{align}\label{inequalityA}
\langle D_\lambda u,v\rangle_H\leq\, C\|u\|_E\|v\|_E,\quad u,v\in W,\,\,\lambda\in X.
\end{align} 
Then the quadratic forms $a_\lambda(u,v)=\langle D_\lambda u,v\rangle_H$ extend uniquely to bounded quadratic forms $\tilde{a}_\lambda$ on $E$, whose Riesz representations are the bounded selfadjoint operators $B_\lambda:E\rightarrow E$ determined by

\begin{align}\label{transform}
\tilde{a}_\lambda(u,v)=\langle B_\lambda u,v\rangle_E,\quad u,v\in E.
\end{align}
We assume that the $B_\lambda$, $\lambda\in X$, are Fredholm operators which satisfy

\begin{align}\label{kernelcond}
\ker( B_\lambda+is K_0)\subset W,\quad (\lambda,s)\in X\times\mathbb{R}.
\end{align}
The following theorem is our main result on the index bundle, which we will prove below in Section \ref{proof:reduction}.

\begin{theorem}\label{theorem:reduction}
Let the spaces $W$, $E$ and $H$ satisfy the condition \eqref{red:comp} and assume that the operators $B=\{B_\lambda\}_{\lambda\in X}$, introduced in \eqref{transform} define a continuous family of Fredholm operators on $E$ such that \eqref{kernelcond} holds. Then
\[\sind(D)=\sind(B)\in K^{-1}(X,Y).\]
\end{theorem}

In the rest of this section we illustrate Theorem \ref{theorem:reduction} with an example. We consider the spaces $H=L^2(I,\mathbb{C}^n)$, $E=H^1_0(I,\mathbb{C}^n)$ and $W=H^2(I,\mathbb{C}^n)\cap H^1_0(I,\mathbb{C}^n)$, where $I=[0,1]$ denotes the unit interval (cf. \cite[\S 2.7]{Masiello}). For a real symmetric matrix $\mathcal{J}$ and a continuous two parameter family of such matrices $S:[a,b]\times I\rightarrow M(n,\mathbb{R})$, we examine the bilinear forms

\begin{align}\label{Hessians}
q_\lambda:E\times E\rightarrow\mathbb{C},\quad 
q_\lambda(u,v)=\int^1_0{\langle\mathcal{J}u'(t),v'(t)\rangle dt}-\int^1_0{\langle S_\lambda(t)u(t),v(t)\rangle dt},
\end{align}
and the operators

\begin{align}\label{JacOps}
D_\lambda:W\rightarrow H,\quad D_\lambda u=\mathcal{J}u''+S_\lambda(\cdot) u,\quad\lambda\in[a,b].
\end{align} 
We see at once that $D$ is a continuous path in $\mathcal{FS}(W,H)$ and so the spectral flow $\sfl(D)$ is defined. Moreover, $q$ is a continuous family of bounded bilinear forms and the associated Riesz representations $B_\lambda$, $\lambda\in[a,b]$, are Fredholm, as is easy to check by using the compactness of the embedding $E\hookrightarrow H$ (cf. \cite[Prop. 3.1]{Pejsachowicz}). Applying integration by parts gives

\[-\langle D_\lambda u,v\rangle_H=q_\lambda (u,v),\quad u\in W,\, v\in E,\,\,\,\lambda\in [a,b],\]
and hence $-D$ and $B$ are related as in \eqref{transform}. In order to apply Theorem \ref{theorem:reduction}, we need to show that the assumptions \eqref{red:comp} and \eqref{kernelcond} hold.\\
From the definition of the scalar product

\[\langle u,v\rangle_E=\int^1_0{\langle u'(t),v'(t)\rangle\,dt},\quad u,v\in E,\]
it may be concluded that the operator

\[K:H\rightarrow E,\quad (Ku)(t)=-\int^t_0{\int^s_0{u(\tau)\,d\tau}ds}+t\int^1_0{\int^s_0{u(\tau)\,d\tau}ds}\]
satisfies $\langle u,v\rangle_H=\langle Ku,v\rangle_E$ for all $u\in H$ and $v\in E$. Given $u\in W$, the second derivative $u''$ belongs to $H$ and from $u(0)=u(1)=0$ it follows that $Ku''=-u$. Consequently, the spaces $W$, $E$ and $H$ fulfil the compatibility condition \eqref{red:comp}. Finally, if we denote by $K_0$ the restriction of $K$ to $H^1_0(I,\mathbb{C}^n)$, then the kernels $\ker(B_\lambda+is\,K_0)$, $(\lambda,s)\in X\times\mathbb{R}$, in \eqref{kernelcond} consist of smooth functions by elliptic regularity. Accordingly, they are contained in $W$ which shows \eqref{kernelcond}. We obtain from Theorem \ref{theorem:reduction} and Proposition \ref{prop:sfl}:

\begin{prop}\label{prop:reduction}
If $D_a$ and $D_b$ are invertible, then

\[\sfl(D)=-\sfl(B).\]
\end{prop} 

The Morse index theorem for geodesics in Riemannian manifolds states the equality of the Morse index of a geodesic and the number of conjugate points along the geodesic counted with multiplicities. For semi-Riemannian manifolds, conjugate points along a geodesic $\gamma$ may accumulate and also the Morse index is no longer finite in general. A generalisation of the Morse index theorem to this setting was obtained by Musso, Pejsachowicz and Portaluri in \cite{Pejsachowicz} as an equality of three integral-valued indices $\mu_{Morse}(\gamma)$, $\mu_{con}(\gamma)$ and $\mu_{spec}(\gamma)$. $\mu_{Morse}(\gamma)$ generalises the Morse index to semi-Riemannian geodesics by means of the spectral flow of a path of bounded selfadjoint Fredholm operators $B$, which are induced by bilinear forms $q$ of the form \eqref{Hessians}. Here $\mathcal{J}$ is a diagonal matrix having $1$ and $-1$ on the diagonal according to the signature of the metric of the underlying manifold, and $S$ contains curvature terms of the manifold along the geodesic. The counting of conjugate points along a geodesic is replaced by $\mu_{con}(\gamma)$, which is defined as a winding number of a planar vector field associated to the Jacobi equation of the geodesic $\gamma$. Finally, in the proof of the theorem appears the third index $\mu_{spec}(\gamma)$, which is the spectral flow of a path of unbounded selfadjoint Fredholm operators as in \eqref{JacOps}. The proof of the semi-Riemannian Morse index theorem in \cite{Pejsachowicz} falls naturally into two parts. The first part shows $\mu_{Morse}(\gamma)=\mu_{spec}(\gamma)$ by using an analytic perturbation argument established by Robbin and Salamon in \cite{Robbin-Salamon} in their construction of the spectral flow. The equality $\mu_{spec}(\gamma)=\mu_{Morse}(\gamma)$, obtained in \cite{Pejsachowicz} by functional analytic arguments, was later proved by the author in \cite{MorseIch} by using the index bundle for families of selfadjoint Fredholm operators and $K$-theory. However, $\mu_{Morse}(\gamma)=\mu_{spec}(\gamma)$ was again only shown by the analytic perturbation argument from \cite{Robbin-Salamon}.  Now the equality $\mu_{Morse}(\gamma)=\mu_{spec}(\gamma)$ follows from Proposition \ref{prop:reduction} and along with \cite{MorseIch} a complete $K$-theoretic proof of the semi-Riemannian Morse index theorem \cite{Pejsachowicz} is established.\\ 
Finally, let us point out that there are several other settings in which Theorem \ref{theorem:reduction} can be applied. We want to mention in particular that, besides periodic Hamiltonian systems which we are studying in this article, also wave equations and noncooperative elliptic systems as considered in \cite{Costa} provide interesting examples.


\section{The index theorem for Hamiltonian systems}\label{section:indthm}
The aim of this section is to state the index theorem for periodic Hamiltonian systems. Throughout we let $X$ be a compact topological space and $Y\subset X$ a closed subspace.\\
At first, we recall the definition of the function spaces that will be used below. The common Hilbert space $L^2(S^1,\mathbb{R}^{2n})$ consists of all functions $u:[0,2\pi]\rightarrow\mathbb{R}^{2n}$ such that

\begin{align}\label{Fourier}
u(t)=c_0+\sum^\infty_{k=1}{a_k\sin\,kt+b_k\cos\,kt},
\end{align}
where $c_0,a_k,b_k\in\mathbb{R}^{2n}$, $k\in\mathbb{N}$, and 

\[\sum^\infty_{k=1}{|a_k|^2+|b_k|^2}<\infty.\]
The scalar product on $L^2(S^1,\mathbb{R}^{2n})$ is given by

\[\langle u,v\rangle_{L^2}=2\pi\langle c_0,\tilde{c}_0\rangle+\pi\sum^\infty_{k=1}{\langle a_k,\tilde{a}_k\rangle+\langle b_k,\tilde{b}_k\rangle},\]
where $\tilde{c}_0$ and $\tilde{a}_k,\tilde{b}_k$ denote the Fourier coefficients of $v\in L^2(S^1,\mathbb{R}^{2n})$. The subset $H^\frac{1}{2}(S^1,\mathbb{R}^{2n})$ of all functions $u\in L^2(S^1,\mathbb{R}^{2n})$ such that

\begin{align}\label{condE}
\sum^\infty_{k=1}{k(|a_k|^2+|b_k|^2)}<\infty
\end{align}
is a Hilbert space in its own right with respect to the scalar product

\begin{align}\label{scalarprodE}
\langle u,v\rangle_{H^\frac{1}{2}}=2\pi\langle c_0,\tilde{c}_0\rangle+\pi\sum^\infty_{k=1}{k(\langle a_k,\tilde{a}_k\rangle+\langle b_k,\tilde{b}_k\rangle)}.
\end{align}
Finally, $H^1(S^1,\mathbb{R}^{2n})$ is defined similarly than $H^\frac{1}{2}(S^1,\mathbb{R}^{2n})$ just by replacing $k$ by $k^2$ in \eqref{condE} and \eqref{scalarprodE}. We will use henceforth that we have compact inclusions \[ H^\frac{1}{2}(S^1,\mathbb{R}^{2n})\hookrightarrow L^2(S^1,\mathbb{R}^{2n})\quad\text{and}\quad H^1(S^1,\mathbb{R}^{2n})\hookrightarrow L^2(S^1,\mathbb{R}^{2n}).\] 
The Hilbert space $H^\frac{1}{2}(S^1,\mathbb{R}^{2n})$ has an orthogonal decomposition $H^\frac{1}{2}(S^1,\mathbb{R}^{2n})=E_+\oplus E_0\oplus E_-$, where

\begin{align*}
E_0&=\left\{u\in H^\frac{1}{2}(S^1,\mathbb{R}^{2n}):\,u\equiv c_0,\,c_0\in\mathbb{R}^{2n}\right\}\\
E_\pm&=\left\{u\in H^\frac{1}{2}(S^1,\mathbb{R}^{2n}):\, u(t)=\sum^\infty_{k=1}{a_k\cos\,kt\mp Ja_k\sin\,kt},\quad a_k\in\mathbb{R}^{2n}\right\}.
\end{align*}
Let $P_\pm$ denote the orthogonal projections in $H^\frac{1}{2}(S^1,\mathbb{R}^{2n})$ onto $E_\pm$ and define a  bilinear form by

\begin{align}\label{Gamma}
\Gamma:H^\frac{1}{2}(S^1,\mathbb{R}^{2n})\times H^\frac{1}{2}(S^1,\mathbb{R}^{2n})\rightarrow\mathbb{R},\quad \Gamma(u,v)=\langle P_+u-P_-u,v\rangle_{H^\frac{1}{2}}.
\end{align}
For $u\in H^1(S^1,\mathbb{R}^{2n})$, we see from the definition of $E_\pm$ and $\langle\cdot,\cdot\rangle_{H^\frac{1}{2}}$ that we can rewrite $\Gamma$ as

\begin{align}\label{Lformula1}
\Gamma(u,v)=\int^{2\pi}_0{\langle Ju',v\rangle\,dt},\quad v\in H^\frac{1}{2}(S^1,\mathbb{R}^{2n}).
\end{align}
Now consider the family of functionals

\begin{align*}
\psi:X\times H^\frac{1}{2}(S^1,\mathbb{R}^{2n})\rightarrow\mathbb{R},\quad\psi(\lambda,u)=\frac{1}{2}\,\Gamma(u,u)+\int^{2\pi}_0{\mathcal{H}(\lambda,t,u(t))\,dt}.
\end{align*}
It is well known (cf. \cite[Prop. 2.1]{Bartsch}) that under the assumptions \eqref{growthcond} each $\psi_\lambda$ is $C^2$ and 

\[\nabla_u\psi(\lambda,u)=(P_+-P_-)u+G(\lambda,u),\]
where

\[\langle G(\lambda,u),v\rangle_{H^\frac{1}{2}}=\int^{2\pi}_0{\langle\nabla_u\mathcal{H}(\lambda,t,u(t)),v(t)\rangle\,dt},\quad u,v\in H^\frac{1}{2}(S^1,\mathbb{R}^{2n}).\]
Consequently, the critical points of $\psi_\lambda$ are precisely the weak solutions of the Hamiltonian system \eqref{equation}. Let us denote by $L_\lambda$, $\lambda \in X$, the linearisation of $\nabla_u\psi(\lambda,u)$ with respect to $u$ at the branch of trivial critical points $X\times\{0\}\subset X\times H^\frac{1}{2}(S^1,\mathbb{R}^{2n})$. Since $\nabla_u\mathcal{H}(\lambda,t,0)=0$, it may be concluded that these operators are given by

\begin{align}\label{Lformula2}
\langle L_\lambda u,v\rangle_{H^\frac{1}{2}}=\Gamma(u,v)+\int^{2\pi}_0{\langle S_\lambda(t)u(t),v(t)\rangle\, dt},\quad u,v\in H^\frac{1}{2}(S^1,\mathbb{R}^{2n}),
\end{align} 
where $S_\lambda(t)=D_u\nabla_u\mathcal{H}(\lambda,t,0)$, $(\lambda,t)\in X\times\mathbb{R}$, as already introduced in \eqref{Smatrix} above. Henceforth, we call the Hamiltonian system \eqref{equation} \textit{admissible} if the operators $L_\lambda$ are invertible for all $\lambda\in Y$.

\begin{lemma}\label{comppert}
The operators $L_\lambda$, $\lambda\in X$, are of the form $L_\lambda=A+K_\lambda$, where $A$ is a selfadjoint Fredholm operator and $K$ is a continuous family of compact operators.  
\end{lemma}

\begin{proof}
At first, we set $A:=P_+-P_-$, which is a selfadjoint Fredholm operator on $H^\frac{1}{2}(S^1,\mathbb{R}^{2n})$. The maps 

\[\beta_\lambda:L^2(S^1,\mathbb{R}^{2n})\times L^2(S^1,\mathbb{R}^{2n})\rightarrow\mathbb{R},\quad (u,v)\mapsto\int^{2\pi}_0{\langle S_\lambda(t)u(t),v(t)\rangle\, dt}\]
restrict to a continuous family of bounded bilinear forms on $H^\frac{1}{2}(S^1,\mathbb{R}^{2n})$. Now

\[\langle K_\lambda u,v\rangle_{H^\frac{1}{2}}=\beta_\lambda(u,v),\quad u,v\in H^\frac{1}{2}(S^1,\mathbb{R}^{2n}),\]
defines a continuous family of bounded selfadjoint operators $K_\lambda$ on $H^\frac{1}{2}(S^1,\mathbb{R}^{2n})$ such that $L_\lambda=A+K_\lambda$, $\lambda\in X$. It remains to show the compactness of $K_\lambda$, $\lambda\in X$. Let $\{u_n\}_{n\in\mathbb{N}}$ and $\{v_n\}_{n\in\mathbb{N}}$ be sequences in $H^\frac{1}{2}(S^1,\mathbb{R}^{2n})$ which converge weakly to some elements $u,v\in H^\frac{1}{2}(S^1,\mathbb{R}^{2n})$. From the compactness of the inclusion $H^\frac{1}{2}(S^1,\mathbb{R}^{2n})\hookrightarrow L^2(S^1,\mathbb{R}^{2n})$, it follows that they converge strongly in $L^2(S^1,\mathbb{R}^{2n})$. Consequently, $\langle K_\lambda u_n,v_n\rangle_{H^\frac{1}{2}}=\beta_\lambda(u_n,v_n)$ converges to $\beta_\lambda(u,v)=\langle K_\lambda u,v\rangle_{H^\frac{1}{2}}$, which shows the compactness of $K_\lambda$ (cf. \cite[\S 21.10]{Zeidler}).  
\end{proof}

The previous lemma implies that $L=\{L_\lambda\}_{\lambda\in X}$ is a continuous family of bounded selfadjoint Fredholm operators in $H^\frac{1}{2}(S^1,\mathbb{R}^{2n})$. We now consider the complexifications $L^\mathbb{C}_\lambda$, of the operators $L_\lambda$, $\lambda\in X$, acting on the complexified Hilbert space $H^\frac{1}{2}(S^1,\mathbb{R}^{2n})^\mathbb{C}=H^\frac{1}{2}(S^1,\mathbb{C}^{2n})$ (cf. eg. \cite[p. 113-114]{Weidmann}). Then $L^\mathbb{C}_\lambda=A^\mathbb{C}+K^\mathbb{C}_\lambda$, $\lambda\in X$, where $A^\mathbb{C}\in\mathcal{FS}(H^\frac{1}{2}(S^1,\mathbb{C}^{2n}))$, $K^\mathbb{C}_\lambda\in\mathcal{K}(H^\frac{1}{2}(S^1,\mathbb{C}^{2n}))$ is selfadjoint, and accordingly 

\[L^\mathbb{C}:X\rightarrow\mathcal{FS}(H^\frac{1}{2}(S^1,\mathbb{C}^{2n}))\]
is a continuous family of bounded selfadjoint Fredholm operators. Of course, $L^\mathbb{C}_\lambda$ is invertible if and only if $L_\lambda$ is invertible, $\lambda\in X$.

\begin{defi}
The \textbf{generalised Morse index} of the admissible family of Hamiltonian systems \eqref{equation} is defined by

\begin{align*}
\mu_{Morse}(\mathcal{H})=\sind(L^\mathbb{C})\in K^{-1}(X,Y).
\end{align*}
\end{defi}


Let us now consider the family of differential operators

\[\mathcal{A}_\lambda:H^1(S^1,\mathbb{R}^{2n})\subset L^2(S^1,\mathbb{R}^{2n})\rightarrow L^2(S^1,\mathbb{R}^{2n}),\quad\mathcal{A}_\lambda u=Ju'+S_\lambda(\cdot)u,\quad\lambda\in X,\]
where as before $S_\lambda(t)=D_u\nabla_u\mathcal{H}(\lambda,t,0)$. The multiplication operators induced by $S$ define a continuous family in $\mathcal{S}(L^2(S^1,\mathbb{R}^{2n}))$, which restricts to a continuous family in \[\mathcal{K}(H^1(S^1,\mathbb{R}^{2n}),L^2(S^1,\mathbb{R}^{2n}))\] by the compactness of the inclusion $H^1(S^1,\mathbb{R}^{2n})\hookrightarrow L^2(S^1,\mathbb{R}^{2n})$. Consequently, 

\[\mathcal{A}_\lambda\in\mathcal{FS}(H^1(S^1,\mathbb{R}^{2n}),L^2(S^1,\mathbb{R}^{2n})),\quad\lambda\in X,\]
and $\mathcal{A}=\{\mathcal{A}_\lambda\}_{\lambda\in X}$ defines a continuous family 

\[\mathcal{A}:X\rightarrow\mathcal{FS}(H^1(S^1,\mathbb{R}^{2n}),L^2(S^1,\mathbb{R}^{2n}))\]
It follows from \eqref{Lformula1} and \eqref{Lformula2} that $\mathcal{A}_\lambda$ is invertible whenever $L_\lambda$ is invertible, $\lambda\in X$. Finally, we complexify the spaces and operators to obtain a continuous family 

\[\mathcal{A}^\mathbb{C}:X\rightarrow\mathcal{FS}(H^1(S^1,\mathbb{C}^{2n}),L^2(S^1,\mathbb{C}^{2n}))\]
of selfadjoint Fredholm operators such that $\mathcal{A}^\mathbb{C}_\lambda$ is invertible for all $\lambda\in Y$ if \eqref{equation} is admissible.

\begin{defi}
The \textbf{spectral index} of the admissible family \eqref{equation} is defined by

\begin{align*}
\mu_{spec}(\mathcal{H})=\sind(\mathcal{A}^\mathbb{C})\in K^{-1}(X,Y).
\end{align*} 
\end{defi}

\begin{rem}\label{rem:specindtriv}
Note that by Remark \ref{rem:trivial}, $\mu_{Morse}(\mathcal{H})$ and $\mu_{spec}(\mathcal{H})$ are trivial if $Y=\emptyset$.
\end{rem}


For our final $K$-theoretic index let 

\[\Psi_z:[0,2\pi]\rightarrow \GL(2n,\mathbb{C}),\quad z=(\lambda,s)\in X\times\mathbb{R},\]
be the fundamental solution of the ordinary differential equation

\begin{equation}\label{ode}
\left\{
\begin{aligned}
J\Psi'_z(t)+S_\lambda(t)\Psi_z(t)+i\,s\Psi_z(t)&=0,\quad t\in[0,2\pi],\\
\Psi_z(0)&=I_{2n}.
\end{aligned}
\right.
\end{equation}
We define a family of matrices by

\begin{align}\label{align:lambda}
\Lambda: X\times\mathbb{R}\rightarrow M(2n,\mathbb{C}),\quad \Lambda(z)=(I_{2n}-\Psi_z(2\pi))^T,
\end{align}
where $\cdot^T$ stands for the transpose of a matrix. Note that $\det\Lambda(z)=0$ if and only if $1$ is a Floquet multiplier of the system \eqref{ode}.

\begin{lemma}\label{detlemma}
For $z=(\lambda,s)\in X\times\mathbb{R}$ the following assertions are equivalent:
\begin{enumerate}
	\item $\det\Lambda(z)=0$,
	\item $s=0$ and $\ker\mathcal{A}^\mathbb{C}_\lambda\neq 0$,
	\item $s=0$ and $\ker\mathcal{A}_\lambda\neq 0$.
\end{enumerate}
\end{lemma}

\begin{proof}
If $\det\Lambda(z)=0$, then there exists $0\neq u_0\in\mathbb{C}^{2n}$ such that $\Psi_z(2\pi)u_0=u_0$. Hence $u(t)=\Psi_z(t)u_0$, $t\in[0,2\pi]$, is a non-trivial element of $H^1(S^1,\mathbb{C}^{2n})$ such that $\mathcal{A}^\mathbb{C}_\lambda u+is\,u=0$. From this we conclude that $s=0$ and $\ker\mathcal{A}^\mathbb{C}_\lambda\neq 0$ since $\mathcal{A}^\mathbb{C}_\lambda$ is selfadjoint and Fredholm. Of course, $\ker(\mathcal{A}^\mathbb{C})=(\ker\mathcal{A})^\mathbb{C}$ and so the second and the third assertion are equivalent. Finally, if $u\in\ker\mathcal{A}_\lambda$ and $z=(\lambda,0)$, then there exists $u_0\in\mathbb{R}^{2n}$ such that $u(t)=\Psi_z(t)u_0$, $t\in[0,2\pi]$. From $u_0=\Psi_z(0)u_0=u(0)=u(2\pi)=\Psi_z(2\pi)u_0$, we obtain $(I_{2n}-\Psi_z(2\pi))u_0=0$ and consequently $\det\Lambda(z)=0$. 
\end{proof}

Let us recall from the definition of the spectral index that $\ker\mathcal{A}_\lambda=0$ for all $\lambda\in Y$ if \eqref{equation} is admissible. Consequently, by the previous lemma, $\Lambda(z)$ is invertible whenever $z=(\lambda,s)$ belongs to $Y\times\mathbb{R}$ or is outside the compact subset $X\times\{0\}$ of $X\times\mathbb{R}$. Now we obtain immediately a relative $K$-theory class as follows:

\begin{defi}
The \textbf{monodromy index} of the admissible family \eqref{equation} is defined by

\begin{align*}
\mu_{mon}(\mathcal{H})=[\Theta(\mathbb{C}^{2n}),\Theta(\mathbb{C}^{2n}),\Lambda]\in K^{-1}(X,Y).
\end{align*}
\end{defi}

Finally, we can state our main theorem, which we prove below in Section \ref{proof:main}.

\begin{theorem}\label{main}
If \eqref{equation} is admissible, then

\begin{align*}
\mu_{Morse}(\mathcal{H})=\mu_{spec}(\mathcal{H})=\mu_{mon}(\mathcal{H})\in K^{-1}(X,Y).
\end{align*}
\end{theorem}

Note that the indices in Theorem \ref{main} are trivial if $Y=\emptyset$ by Remark \ref{rem:specindtriv}. Non-vanishing indices can be obtained, for instance, if $X$ is a compact interval and $Y=\partial X$ is its boundary. We devote the following section to this special case of Theorem \ref{main}.


\section{The special case of paths}\label{section:paths}
In this section we discuss Theorem \ref{main} in the case that the families of operators are paths. Let us recall at first the definition of the \textit{Conley-Zehnder index} for paths of symplectic matrices, where we follow the approach of Robbin and Salamon in \cite{Robbin-Salamon-Maslov}:\\
An $n$-dimensional subspace $\ell\subset V$ of a (real) symplectic vector space $(V,\omega)$ of dimension $2n$ is called \textit{Lagrangian} if the restriction of the symplectic form $\omega$ to $\ell$ vanishes. The set $\Lambda(V)$ of all Lagrangian subspaces of $V$ forms a smooth submanifold of the ordinary Grassmannian of all $n$-dimensional subspaces of $V$, which is called the \textit{Lagrangian Grassmannian}. We fix a Lagrangian subspace $\ell_0\in\Lambda(V)$ of $V$. The \textit{Maslov index} $m(\gamma,\ell_0)$ provides an integer-valued homotopy invariant of paths $\gamma:[a,b]\rightarrow\Lambda(V)$ whose endpoints are transverse to $\ell_0$. Roughly speaking, it is the number of non-transverse intersections of $\gamma([a,b])$ with $\ell_0$. There are several approaches to the construction of this invariant and we refer to \cite{Cappell} for a detailed exposition. Let us now restrict our discussion to the symplectic space $(V,\omega)$, where $V=\mathbb{R}^{2n}\times\mathbb{R}^{2n}$ and $\omega$ is the symplectic form induced by 

\[\begin{pmatrix}
-J&0\\
0&J
\end{pmatrix}.\]
Here $J$ denotes as before the standard symplectic matrix \eqref{standardmatrix}. Given a symplectic matrix $M\in\Sp(2n,\mathbb{R})$, its graph 

\[\gra(M)=\{(u,M u)\in\mathbb{R}^{2n}\times\mathbb{R}^{2n}:\,u\in\mathbb{R}^{2n}\}\]
is a Lagrangian subspace of $(V,\omega)$. Of course, the diagonal $\Delta\subset\mathbb{R}^{2n}\times\mathbb{R}^{2n}$ is a Lagrangian subspace of $(V,\omega)$ as well, and $\gra(M)\cap\Delta\neq\{0\}$ if and only if $\det(I_{2n}-M)=0$. Consequently, if $M:[a,b]\rightarrow\Sp(2n,\mathbb{R})$ is a path such that $\det(I_{2n}-M_a)$ and $\det(I_{2n}-M_b)$ do not vanish, then the endpoints of the path $\gra(M):[a,b]\rightarrow\Lambda(V)$ are transverse to $\Delta\subset\mathbb{R}^{2n}\times\mathbb{R}^{2n}$. The \textit{Conley-Zehnder index} of $M$ is defined by

\begin{align*}
\mu_{CZ}(M)=m(\gra(M),\Delta)\in\mathbb{Z}.
\end{align*}
Now assume that $X=[a,b]$ parametrises the family of Hamiltonian systems \eqref{equation} and $Y=\partial X=\{a,b\}$. The initial value problems

\begin{equation}\label{odepath}
\left\{
\begin{aligned}
J\Psi'_\lambda(t)+S_\lambda(t)\Psi_\lambda(t)&=0,\quad t\in[0,2\pi],\\
\Psi_\lambda(0)&=I_{2n},
\end{aligned}
\right.
\end{equation}
induce a path $M:[a,b]\rightarrow\Sp(2n,\mathbb{R})$ of symplectic matrices by $M_\lambda=\Psi_\lambda(2\pi)$. Note that \eqref{odepath} coincides with equation \eqref{ode} from the definition of the monodromy index if $s=0$. If \eqref{equation} is admissible, then $I_{2n}-\Psi_a(2\pi)$ and $I_{2n}-\Psi_b(2\pi)$ are invertible because of Lemma \ref{detlemma} and hence the Conley-Zehnder index of $M$ is defined. Note that, heuristically, $\mu_{CZ}(M)$ counts instants $\lambda\in X$ for which $\gra(\Psi_\lambda(2\pi))\cap\Delta\neq\{0\}$, i.e. $\lambda\in X$ for which the linear Hamiltonian system \eqref{diffop} admits non-trivial periodic solutions.\\ 
Let us now investigate the three indices of Theorem \ref{main} for $X=[a,b]$ and $Y=\partial X$. The families of operators $L$ and $\mathcal{A}$ from Section \ref{section:indthm} are then paths of selfadjoint Fredholm operators acting on the real Hilbert spaces $H^\frac{1}{2}(S^1,\mathbb{R}^{2n})$ and $L^2(S^1,\mathbb{R}^{2n})$, respectively. Accordingly, the spectral flows $\sfl(L)$ and $\sfl(\mathcal{A})$ are defined, and if \eqref{equation} is admissible, then $L$ and $\mathcal{A}$ have invertible endpoints. Since the spectral flow is invariant under complexification by \eqref{invsfl}, we obtain from Proposition \ref{prop:sfl}

\[\sfl(L)=c_1(\mu_{Morse}(\mathcal{H}))\quad\text{and}\quad \sfl(\mathcal{A})=c_1(\mu_{spec}(\mathcal{H})).\]
Let us recall from \eqref{Chern} that the first Chern number $c_1$ for an element in $K^{-1}(X,Y)$ of the form $[\Theta(\mathbb{C}^{2n}),\Theta(\mathbb{C}^{2n}),\tilde{a}]$ is given by

\[c_1([\Theta(\mathbb{C}^{2n}),\Theta(\mathbb{C}^{2n}),\tilde{a}])=w(\det(\tilde{a})\circ\kappa,0),\]
when $\kappa:S^1\rightarrow X\times\mathbb{R}$ is any simple positively oriented curve surrounding the support of $\tilde{a}$. From this we get for the monodromy index

\[c_1(\mu_{mon}(\mathcal{H}))=w(\rho\circ\kappa,0),\]
where $\rho$ is the planar vector field on $X\times\mathbb{R}\subset\mathbb{C}$ defined by

\[\rho:X\times\mathbb{R}\rightarrow\mathbb{C},\quad \rho(z)=\det\Lambda(z)\]
and $\Lambda:X\times\mathbb{R}\rightarrow M(2n,\mathbb{C})$ is the matrix family introduced in \eqref{align:lambda}. Note that according to Lemma \ref{detlemma}, the support of $\mu_{mon}(\mathcal{H})$ is contained in $(a,b)\times\{0\}$. Finally, it is proved in \cite[Prop. 2.1]{SFLPejsachowiczII} that the spectral flow of $L$ is equal to the Conley-Zehnder index of the path $M$ in \eqref{odepath}. In summary, we obtain from Theorem \ref{main} the following result.

\begin{theorem}\label{mainpath}
If \eqref{equation} is admissible, then

\[\sfl(L)=\sfl(\mathcal{A})=w(\rho\circ\kappa,0)=\mu_{CZ}(M)\in\mathbb{Z},\]
where $\kappa:S^1\rightarrow X\times\mathbb{R}$ is any simple positively oriented curve surrounding $(a,b)\times\{0\}$.
\end{theorem}

\begin{rem}
The equality of the spectral flow of $\mathcal{A}$ and the Conley-Zehnder index of $M$ was originally proven by Salamon and Zehnder in Theorem 3 of \cite{Salamon-Zehnder}.
\end{rem}

The definitions of the Conley-Zehnder index of the path $M$ and the winding number of the planar vector field $\rho$ in Theorem \ref{mainpath} seem to us rather different in nature, although also the definition of $\mu_{CZ}(M)$ involves a winding number (cf. \cite{Salamon-Zehnder}). We want to point out that there is a well studied theory for computing winding numbers of planar vector fields (cf. eg. \cite{KrasnoselskiiI}). As an example, let us assume that the components of $S_\lambda(t)$ in \eqref{Smatrix} depend analytically on $(\lambda,t)\in(a,b)\times\mathbb{R}$. Then $\rho$ is analytic as well and hence its zeroes are a finite subset of $(a,b)\times\{0\}$. From the properties of the winding number it is clear that we can assume without loss of generality that there exists only a single zero, which is then of the form $z_0=(\lambda_0,0)\in (a,b)\times\mathbb{R}$. Let $\kappa:S^1\rightarrow X\times\mathbb{R}$ be a simple positively oriented curve surrounding $z_0$.\\
We have in a neighbourhood of $z_0$

\begin{align*}
\rho(z)=\rho(\lambda,s)=(P(\lambda-\lambda_0,s)+f(\lambda-\lambda_0,s))+i(Q(\lambda-\lambda_0,s)+g(\lambda-\lambda_0,s)),
\end{align*}
where $P$ and $Q$ are real homogenous polynomials of degree $m$ and $n$, respectively, and $|f(z)|=O(|z|^{m+1})$, $|g(z)|=O(|z|^{n+1})$. Set $\eta:=P+iQ$ and let us require henceforth that $z_0$ is the only zero of $\eta$.  We see at once that 

\[w(\rho\circ\kappa,0)=w(\eta\circ\kappa,0),\]
and now the winding number of $\eta\circ\kappa$ can be computed directly from the coefficients of $\eta$ as follows: we assume without loss of generality that $m\geq n$ and set $N_0(s):=P(1,s)$, $N_1(s):=Q(1,s)$. Let $N_0,N_1,\ldots,N_l$ be non-trivial polynomials such that $N_{i+1}(s)$ is the rest of the division of $N_{i-1}(s)$ by $N_i(s)$ so that, incidentally, $N_l(s)$ is the greatest common divisor of $N_0$ and $N_1$ by the Euclidean algorithm. For any $r\in\mathbb{R}$ which is not a zero of any of the polynomials $N_i$, we denote by $m(r)$ the number of sign changes in the sequence of integers $N_0(r),\ldots,N_l(r)$. Since $m(r)$ is defined and constant if $r$ or $-r$ is sufficiently large, the limits $m_+$ and $m_-$ of $m(r)$ for $r\rightarrow\pm\infty$ exist. Now Theorem 10.2 of \cite{KrasnoselskiiII} provides the following simple formula for the indices in Theorem \ref{mainpath}:

\begin{align}\label{analytic}
w(\rho\circ\kappa,0)=(1+(-1)^{m+n})\,\frac{m_+-m_-}{2}.
\end{align}



\section{Bifurcation}
In this section we consider the generally nonlinear Hamiltonian systems \eqref{equation}, where we no longer assume the parameter space $X$ to be compact. Note that $u\equiv 0$ solves all these equations and our aim is to study multiparameter bifurcation from this branch of trivial periodic solutions by using Theorem \ref{main}. In what follows we assume as in Section \ref{section:indthm} that \eqref{equation} is admissible, i.e. the operators $L_\lambda$ introduced in \eqref{Lformula2} are invertible for all $\lambda\in Y$.

\begin{defi}
We call $\lambda^\ast\in X$ a bifurcation point of periodic solutions for the Hamiltonian system \eqref{equation} if every neighbourhood of $(\lambda^\ast,0)$ in $X\times H^\frac{1}{2}(S^1,\mathbb{R}^{2n})$ contains elements $(\lambda,u)$, where $u\neq 0$ is a solution of \eqref{equation} for the parameter value $\lambda$.
\end{defi}

We denote by $B(\mathcal{H})\subset X$ the set of all bifurcation points of the family \eqref{equation} and we observe that $B(\mathcal{H})$ is closed by the very definition of a bifurcation point.\\
In \cite{AleBifIch} the author has studied in collaboration with A. Portaluri several parameter bifurcation of critical points for abstract families of functionals by using the main theorem of \cite{SFLPejsachowicz}. For the optimal results we need to assume simply connected parameter spaces. Here we want to use the subsequent article \cite{JacBifIch}, written in collaboration with J. Pejsachowicz, and our Theorem \ref{main} in order to investigate several parameter bifurcation for periodic solutions of Hamiltonian systems as critical points of the functional $\psi$ in Section \ref{section:indthm}. Interestingly, it turns out that we do not need to require simple connectedness of the parameter space $X$ any longer due to the special form of the operators $L$ from Section \ref{section:indthm}.\\
Let us recall that in the definition of the monodromy index in \eqref{align:lambda}, we defined a matrix family $\Lambda:X\times\mathbb{R}\rightarrow M(2n,\mathbb{C})$ by using the monodromy matrices of the family of equations \eqref{ode}. Let $I=[a,b]\subset\mathbb{R}$ denote a compact interval. Given a path $\gamma:(I,\partial I)\rightarrow (X,Y)$, we obtain by composition a map

\[\Lambda\circ(\gamma,id):I\times\mathbb{R}\rightarrow M(2n,\mathbb{C})\]
and an integer

\[d(\gamma):=w(\det(\Lambda\circ(\gamma,id))\circ \kappa,0)\in\mathbb{Z},\]
where as before $\kappa:S^1\rightarrow I\times\mathbb{R}$ is any simple positively oriented path surrounding $(a,b)\times\{0\}$ and $w(\cdot,0)$ denotes the winding number with respect to $0$ for closed curves in $\mathbb{C}\setminus\{0\}$. Since $\det(\Lambda(\gamma(t),s))\neq 0$ for all $(t,s)\notin (a,b)\times\{0\}$ by Lemma \ref{detlemma}, $d(\gamma)$ does not depend on the choice of $\kappa$. Note furthermore that $d(\gamma)=\mu_{CZ}(M\circ\gamma)$ by Theorem \ref{mainpath}, where $M:X\rightarrow\Sp(2n,\mathbb{R})$ is the associated family of monodromy matrices of the equations \eqref{odepath}.

\begin{theorem}\label{theorem:bifurcation}
Assume that the admissible family of Hamiltonian systems \eqref{equation} is parametrised by a path-connected topological space $X$ and let $\emptyset\neq Y\subset X$ be a closed subspace.
\begin{enumerate}
	\item[i)] If there exists a path $\gamma:(I,\partial I)\rightarrow (X,Y)$ such that $d(\gamma)\neq 0$, then $B(\mathcal H)$ disconnects $X$.
	\item[ii)] If there exists a sequence of paths $\gamma_n:(I,\partial I)\rightarrow(X,Y)$, $n\in\mathbb{N}$, such that 
	
	\begin{align*}
	\lim_{n\rightarrow\infty}|d(\gamma_n)|=\infty,
	\end{align*} 
	then $X\setminus B(\mathcal H)$ has infinitely many path components.
\end{enumerate}
\end{theorem}

Intuitively, part i) of Theorem \ref{theorem:bifurcation} states that $B(\mathcal H)$ is a subset of codimension $1$ in $X$. Let us recall that the \textit{covering dimension} $\dim\mathcal{X}$ of a topological space $\mathcal{X}$ is the minimal value of $n\in\mathbb{N}$ such that every finite open cover of $\mathcal{X}$ has a finite open refinement in which no point is included in more than $n+1$ elements. Now, from the disconnectedness of its complement in $X$, it indeed follows from Theorem \ref{theorem:bifurcation} that the covering dimension of $B(\mathcal H)$ is at least $\dim(X)-1$ if $X$ is a topological manifold (cf. \cite[Prop. V.6]{Dimension}).\\
Finally, setting $(X,Y)=(I,\partial I)$, we obtain from the first assertion of Theorem \ref{theorem:bifurcation} that if one of the integers in Theorem \ref{mainpath} is non-zero, then any neighbourhood of $I\times\{0\}$ in $I\times H^\frac{1}{2}(S^1,\mathbb{R}^{2n})$ contains non-trivial solutions of \eqref{equation}. This result was proved in \cite[Theorem 2.2]{SFLPejsachowiczII} under the additional assumption that $\mathcal{H}$ depends smoothly on the parameter $\lambda\in I$.


\section{Proofs of the theorems}\label{section:proofs}
In this section we give the proofs of Theorem \ref{theorem:reduction}, Theorem \ref{main} and Theorem \ref{theorem:bifurcation}.

\subsection{Proof of Theorem \ref{theorem:reduction}}\label{proof:reduction}
Let us recall at first that $\sind(B)=\ind(\overline{B})$ and $\sind(D)=\ind(\overline{D})$, where

\begin{align*}
\overline{B}_{(\lambda,s)}=B_\lambda+is\,I_E,\quad \overline{D}_{(\lambda,s)}=D_\lambda+is\, I_H,\quad(\lambda,s)\in X\times\mathbb{R}.
\end{align*}
We divide the proof of Theorem \ref{theorem:reduction} into three parts. In the first part we introduce a further family of Fredholm operators $\tilde{B}$ parametrised by $X\times\mathbb{R}$. In the subsequent second and third part we show that $\ind(\tilde{B})=\sind(B)$ and $\ind(\tilde{B})=\sind(D)$, respectively.

\subsubsection*{Step 1}
We define a family of bounded operators by

\begin{align*}
\tilde{B}:X\times\mathbb{R}\rightarrow\mathcal{L}(E),\quad\tilde{B}_{(\lambda,s)}u=B_\lambda u+is\,K_0u,
\end{align*}
where $K_0$ is the compact nonnegative operator introduced in \eqref{scalarK0}. Since $\tilde{B}_{(\lambda,s)}$ is a compact perturbation of $B_\lambda$, it is a Fredholm operator of index $0$. Now consider for $(\lambda,s)\in X\times\mathbb{R}$ the diagram

\begin{align}\label{diagram}
\begin{split}
\xymatrix{
E\ar[rr]^{\tilde{B}_{(\lambda,s)}}&&E\\
W\ar[u]^\iota\ar[rr]^{\overline{D}_{(\lambda,s)}}&&H\ar[u]_K
}
\end{split}
\end{align}
where $\iota:W\rightarrow E$ denotes the continuous inclusion. The diagram commutes because of

\begin{align*}
\langle B_\lambda u+is\,K_0u,v\rangle_E&=\langle B_\lambda u,v\rangle_E+is\langle K_0u,v\rangle_E=\langle D_\lambda u+is\,u,v\rangle_H\\
&=\langle K(D_\lambda u+is\,u),v\rangle_E,\quad u\in W,\,v\in E.
\end{align*}
The next lemma deals with the relationship between the families $\overline{B}$ and $\tilde{B}$.

\begin{lemma}\label{lemma:proof}
For $(\lambda,s)\in X\times\mathbb{R}$ the following assertions are equivalent:

\begin{enumerate}
	\item[i)] $\ker\tilde{B}_{(\lambda,s)}\neq 0$,
	\item[ii)] $\ker\overline{B}_{(\lambda,s)}\neq 0$,
	\item[iii)] $s=0$ and $\ker B_\lambda\neq 0$. 
\end{enumerate}

\end{lemma}

\begin{proof}
By using the selfadjointness of $B_\lambda$, it is clear that ii) and iii) are equivalent. Since iii) obviously implies i), it is sufficient to show that the first assertion implies the third one. Assume that $\tilde{B}_{(\lambda,s)}u=B_\lambda u+is\,K_0u=0$ for some $0\neq u\in E$. Then $u\in W$ by \eqref{kernelcond}, and since \eqref{diagram} is commutative and $K$ injective, we obtain $D_\lambda u+is\,u=0$. From the selfadjointness of $D_\lambda$, we conclude that $s=0$ and $u\in\ker D_\lambda$. Finally, by using again the commutativity of \eqref{diagram}, we see that $u\in\ker B_\lambda$. 
\end{proof} 

From the previous lemma we conclude that $\tilde{B}_{(\lambda,s)}$ is invertible for all $(\lambda,s)\notin X\times\{0\}$ and all $(\lambda,0)\in Y\times\{0\}$. Hence \[\ind\tilde{B}\in K^{-1}(X,Y)\] is defined.


\subsubsection*{Step 2}
We define a path of bounded selfadjoint operators on $E$ by $A(t)=(1-t)K_0+t\,I_E$, $t\in[0,1]$. Since $K_0$ is nonnegative, $A(t)$ is invertible for all $t\in(0,1]$. Hence there exists for any $t\in(0,1]$ a unique square root $A(t)^\frac{1}{2}$, which is an invertible and selfadjoint operator on $E$. We now consider the homotopy of Fredholm operators

\begin{align*}
h:[0,1]\times(X\times\mathbb{R}\times E)\rightarrow E,\quad h(t)_{(\lambda,s)}u=B_\lambda u+is\, A(t)u
\end{align*}
and note that $h(0)=\tilde{B}$ and $h(1)=\overline{B}$.

\begin{lemma}
$h(t)_{(\lambda,s)}$, $t\in[0,1]$, $(\lambda,s)\in X\times\mathbb{R}$, is invertible if $s\neq 0$.
\end{lemma}

\begin{proof}
Note at first that $h(0)_{(\lambda,s)}=\tilde{B}_{(\lambda,s)}$ is invertible if $s\neq 0$ according to Lemma \ref{lemma:proof}. Hence we can restrict to the case $t>0$ and obtain

\[h(t)_{(\lambda,s)}=B_\lambda+is\,A(t)=A(t)^{\frac{1}{2}}(A(t)^{-\frac{1}{2}}B_\lambda A(t)^{-\frac{1}{2}}+is\,I_E)A(t)^{\frac{1}{2}}.\]
Since $A(t)^{-\frac{1}{2}}B_\lambda A(t)^{-\frac{1}{2}}$ is selfadjoint and $s\neq 0$, we conclude that $A(t)^{-\frac{1}{2}}B_\lambda A(t)^{-\frac{1}{2}}+is\,I_E$ and, consequently, also $h(t)_{(\lambda,s)}$ is invertible.
\end{proof}

From the definition of $h$ it is clear that $h(t)_{(\lambda,s)}$ is also invertible for all $t\in[0,1]$ if $(\lambda,s)\in Y\times\{0\}$. Now the homotopy invariance property ii) in Section \ref{section:indbundle} shows that

\[\sind(B)=\ind(\overline{B})=\ind(\tilde{B})\in K^{-1}(X,Y).\]


\subsubsection*{Step 3}
Since the operators $B_\lambda$ are selfadjoint and Fredholm, we have for each $\lambda\in X$ an orthogonal decomposition $E=\im(B_\lambda)\oplus\ker(B_\lambda)$. It is readily seen from the proof of Lemma \ref{transversal} that there exist $\lambda_1,\ldots,\lambda_N\in X$ such that $U:=\bigoplus_{i=1}^N{\ker B_{\lambda_i}}$ is transversal to the image of the family $B$ as in \eqref{transversality}, i.e.

\begin{align*}
\im(B_\lambda)+U=E,\quad \lambda\in X.
\end{align*}
Since $\tilde{B}_{(\lambda,s)}$ is invertible for $s\neq 0$ according to Lemma \ref{lemma:proof}, we conclude that 

\begin{align*}
\im(\tilde{B}_{(\lambda,s)})+U=E,\quad (\lambda,s)\in X\times\mathbb{R}.
\end{align*}
From the definition of $U$, \eqref{red:comp} and \eqref{kernelcond} we see that $U\subset W\subset\im K$ and hence there exists a finite dimensional subspace $V_1\subset H$ such that $K(V_1)\supset U$. Now choose $V_2\subset H$ finite dimensional such that

\begin{align*}
\im(\overline{D}_{(\lambda,s)})+V_2=H,\quad (\lambda,s)\in X\times\mathbb{R},
\end{align*} 
and set $V:=V_1+V_2$. We have by definition

\begin{align*}
\ind(\tilde{B})&=[E(\tilde{B},K(V)),\Theta(K(V)),\tilde{B}\mid_{E(\tilde{B},K(V))}],\quad 
\sind(D)=[E(\overline{D},V),\Theta(V),\overline{D}\mid_{E(\overline{D},V)}].
\end{align*}
It follows from the commutativity of \eqref{diagram} that the inclusion $\iota:W\rightarrow E$ maps $E(\overline{D},V)$ into $E(\tilde{B},K(V))$. Since
\[\dim E(\overline{D},V)=\dim V=\dim K(V)=\dim E(\tilde{B},K(V)),\] 
where we use \eqref{dimE} and the injectivity of $K$, $\iota$ actually induces an isomorphism. Finally, by employing again the commutativity of \eqref{diagram}, we obtain that the diagram

\begin{align*}
\xymatrix{
E(\tilde{B},K(V))\ar[rr]^{\tilde{B}}&&\Theta(K(V))\\
E(\overline{D},V)\ar[rr]^{\overline{D}}\ar[u]^\iota&&\Theta(V)\ar[u]_K
}
\end{align*}
commutes. Consequently, $\sind(D)=\ind(\tilde{B})$ by the definition of $K$-theory, and the proof of Theorem \ref{theorem:reduction} is complete.


\subsection{Proof of Theorem \ref{main}}\label{proof:main}
For the proof of Theorem \ref{main} we set in accordance with the notation in Section \ref{section:reduction}

\[W=H^1(S^1,\mathbb{C}^{2n}),\quad E=H^\frac{1}{2}(S^1,\mathbb{C}^{2n}),\quad H=L^2(S^1,\mathbb{C}^{2n}).\]


\subsubsection*{Part I: $\mu_{Morse}(\mathcal{H})=\mu_{spec}(\mathcal{H})$}
The first part of our proof is based on Theorem \ref{theorem:reduction}. We note at first that by \eqref{Lformula1} and \eqref{Lformula2}

\begin{align*}
\langle\mathcal{A}^\mathbb{C}_\lambda u,v\rangle_{H}=\langle L^\mathbb{C}_\lambda u,v\rangle_{E},\quad u\in W,\,\,v\in E,
\end{align*}
and consequently the families $\mathcal{A}^\mathbb{C}:X\rightarrow\mathcal{FS}(W,H)$ and $L^\mathbb{C}:X\rightarrow\mathcal{FS}(E)$ are in the relation assumed in Section \ref{section:reduction}.\\
We define $K:H\rightarrow E$ as the unique bounded selfadjoint operator such that

\begin{align}\label{align:proof}
\langle u,v\rangle_{H}=\langle Ku,v\rangle_{E},\quad u\in H,\,\, v\in E,
\end{align}
and denote by $K_0$ the restriction of $K$ to $E$. It is readily seen that

\begin{align*}
Ku=c_0+\sum^\infty_{k=1}{\frac{a_k}{k}\sin\,kt+\frac{b_k}{k}\cos\,kt},
\end{align*}
where $c_0,a_k,b_k\in\mathbb{C}^{2n}$, $k\in\mathbb{N}$, denote the Fourier coefficients of $u\in H$. Hence $\im K=W$, which shows \eqref{red:comp}.\\
Finally, let us assume that $u\in\ker(L^\mathbb{C}_\lambda+is\,K_0)$ for some $(\lambda,s)\in X\times\mathbb{R}$. We obtain from \eqref{align:proof}

\begin{align*}
0&=\langle L^\mathbb{C}_\lambda u,v\rangle_{E}+is\langle K_0u,v\rangle_{E}=\Gamma^\mathbb{C}(u,v)+\int^{2\pi}_0{\langle S_\lambda(t)u(t),v(t)\rangle dt}+is\langle u,v\rangle_{H}\\
&=\Gamma^\mathbb{C}(u,v)+\int^{2\pi}_0{\langle(S_\lambda(t)+is\,I_{2n})u(t),v(t)\rangle dt},\quad u,v\in E,
\end{align*}
where $\Gamma^\mathbb{C}$ denotes the complexification of the bilinear form $\Gamma$ introduced in \eqref{Gamma}. Consequently, $u$ is a weak solution of

\begin{equation*}
\left\{
\begin{aligned}
Ju'(t)+S_\lambda(t)&u(t)+is\,u(t)=0,\quad t\in [0,2\pi],\\
u(0)&=u(2\pi).
\end{aligned}
\right.
\end{equation*}
By a well known regularity argument (cf. \cite[\S 6]{Rabinowitz}), every weak solution of this equation is in fact a classical solution and we obtain 

\begin{align*}
\ker(L^\mathbb{C}_\lambda+is\,K_0)\subset C^1(S^1,\mathbb{C}^{2n})\subset W.
\end{align*}
Accordingly, assumption \eqref{kernelcond} is satisfied as well and we conclude from Theorem \ref{theorem:reduction} that

\begin{align*}
\mu_{Morse}(\mathcal{H})=\mu_{spec}(\mathcal{H})\in K^{-1}(X,Y).
\end{align*}


\subsubsection*{Part II: $\mu_{spec}(\mathcal{H})=\mu_{mon}(\mathcal{H})$}
The second part of our proof of Theorem \ref{main} is less direct than the first one.\\
We will use throughout the paths of matrices $\Psi_z:[0,2\pi]\rightarrow \GL(2n,\mathbb{C})$, $z=(\lambda,s)\in X\times\mathbb{R}$, which we introduced in \eqref{ode}.

\begin{lemma}
The space

\begin{align*}
\mathcal{E}=\{\Psi^{-1}_z(\cdot)u:\,u\in W,\, z=(\lambda,s)\in X\times\mathbb{R}\}
\end{align*}
is a Hilbert subbundle of the product $(X\times\mathbb{R})\times H^1([0,2\pi],\mathbb{C}^{2n})$.
\end{lemma}

\begin{proof}
We have a bundle isomorphism

\begin{align*}
\varphi:(X\times\mathbb{R})\times H^1([0,2\pi],\mathbb{C}^{2n})\rightarrow (X\times\mathbb{R})\times H^1([0,2\pi],\mathbb{C}^{2n}),\quad (\varphi_zu)(t)=\Psi^{-1}_z(t)\,u(t),
\end{align*}
and $W\subset H^1([0,2\pi],\mathbb{C}^{2n})$ is a closed subspace. Hence

\begin{align*}
\mathcal{E}=\varphi((X\times\mathbb{R})\times W)\subset (X\times\mathbb{R})\times H^1([0,2\pi],\mathbb{C}^{2n}) 
\end{align*}
is a Hilbert subbundle.
\end{proof}

We now define isomorphisms

\begin{align*}
M:&\,\mathcal{E}\rightarrow(X\times\mathbb{R})\times W,\quad (M_zu)(t)=\Psi_z(t)\,u(t),\\
M^T:&\,(X\times\mathbb{R})\times H\rightarrow(X\times\mathbb{R})\times H,\quad (M^T_zu)(t)=\Psi^T_z(t)\,u(t),
\end{align*}
and consider the compositions

\begin{align*}
T_z:=M^T_z\circ\overline{\mathcal{A}^\mathbb{C}}_{z}\circ M_z:\mathcal{E}_z\rightarrow H,\quad z=(\lambda,s)\in X\times\mathbb{R}.
\end{align*}
Then $T:\mathcal{E}\rightarrow (X\times\mathbb{R})\times H$ is a Fredholm morphism between Hilbert bundles and we obtain from the properties i) and iii) in Section \ref{section:indbundle}

\begin{align*}
\ind(T)=\ind(M^T\circ\overline{\mathcal{A}^\mathbb{C}}\circ M)=\ind(M^T)+\ind(\overline{\mathcal{A}^\mathbb{C}})+\ind(M)=\sind(\mathcal{A}^\mathbb{C}).
\end{align*}
Our aim is now to compute $\ind(T)$. At first, we obtain for $u\in\mathcal{E}_z$, $z=(\lambda,s)\in X\times\mathbb{R}$,

\begin{align*}
(T_zu)(t)&=(M^T_z(\mathcal{A}^\mathbb{C}_\lambda+is\,I_{H})M_zu)(t)=\Psi^T_z(\mathcal{A}^\mathbb{C}_\lambda+is\,I_{H})(\Psi_zu)(t)\\
&=\Psi^T_z(t)(J\Psi'_z(t)u(t)+J\Psi_z(t)u'(t)+S_\lambda(t)\Psi_z(t)u(t)+is\Psi_z(t)u(t))\\
&=\Psi^T_z(t)(-S_\lambda(t)\Psi_z(t)u(t)-is\Psi_z(t)u(t)+J\Psi_z(t)u'(t)+S_\lambda(t)\Psi_z(t)u(t)+is\Psi_z(t)u(t))\\
&=\Psi^T_z(t)J\Psi_z(t)u'(t)=Ju'(t),
\end{align*}
where the last equality is easily seen by differentiating $\Psi^T_z(t)J\Psi_z(t)$ with respect to $t$.\\
We have a decomposition $H=Y_1\oplus Y_2$, where $Y_1$ denotes the space of constant functions and

\begin{align*}
Y_2=\left\{y\in L^2(S^1,\mathbb{C}^{2n}):\,\int^{2\pi}_{0}{y(t)\, dt}=0\right\}.
\end{align*} 
Let $y\in Y_2$ be given. The function $w(t)=-J\int^t_0{y(s)\,ds}$, $t\in[0,2\pi]$, belongs to $W$. From $w(0)=w(2\pi)=0$ we see moreover that $w\in\mathcal{E}_z$ for all $z\in X\times\mathbb{R}$ and $T_zw=y$. Hence $Y_1$ is transverse to the image of the Fredholm morphism $T$ as in \eqref{transversality} and we obtain

\begin{align*}
\ind(T)=[E(T,Y_1),\Theta(Y_1),T\mid_{E(T,Y_1)}]\in K^{-1}(X,Y).
\end{align*}
Now we have

\begin{align*}
E(T,Y_1)_z&=\{u\in\mathcal{E}_z:T_zu\in Y_1\}=\{u\in\mathcal{E}_z:T_zu\,\,\text{constant}\,\}\\
&=\{(2\pi-t)a+tb\in\mathcal{E}_z:\,a,b\in\mathbb{C}^{2n}\}.
\end{align*} 
By definition of $\mathcal{E}_z$, we conclude that $(2\pi-t)a+tb\in\mathcal{E}_z$ if and only if $(2\pi-t)a+tb=\Psi(t)^{-1}u(t)$, $t\in[0,2\pi]$, for some function $u\in W$. From the periodicity of $u$, it follows that $b=\Psi_z(2\pi)^{-1}a$, and hence

\begin{align*}
E(T,Y_1)_z=\{(2\pi-t)a+t\Psi_z(2\pi)^{-1}a:\,a\in\mathbb{C}^{2n}\}.
\end{align*} 
Now we define bundle isomorphisms

\begin{align*}
&E(T,Y_1)\rightarrow\Theta(\mathbb{C}^{2n}),\quad u\mapsto\frac{1}{2\pi} u(0),\\
&\Theta(Y_1)\rightarrow\Theta(\mathbb{C}^{2n}),\quad v\mapsto v(0),
\end{align*}
and a bundle morphism

\begin{align*}
N:\Theta(\mathbb{C}^{2n})\rightarrow\Theta(\mathbb{C}^{2n}),\quad N_za=J(\Psi_z(2\pi)^{-1}-I_{2n})a.
\end{align*}
We obtain a commutative diagram

\begin{align*}
\xymatrix{
E(T,Y_1)\ar[r]^T\ar[d]&\Theta(Y_1)\ar[d]\\
\Theta(\mathbb{C}^{2n})\ar[r]^N&\Theta(\mathbb{C}^{2n})
}
\end{align*}
and conclude from the definition of the $K$-theory groups that 

\begin{align*}
\ind(T)=[\Theta(\mathbb{C}^{2n}),\Theta(\mathbb{C}^{2n}),N]\in K^{-1}(X,Y).
\end{align*}
Finally, note that

\begin{align*}
N_z&=J(\Psi_z(2\pi)^{-1}-I_{2n})=J(-J\Psi_z(2\pi)^TJ-I_{2n})=\Psi^T_z(2\pi)J-J=(\Psi_z(2\pi)-I_{2n})^TJ
\end{align*}
and, by deforming $J$ inside $\GL(2n,\mathbb{C})$ to $-I_{2n}$, we obtain from Lemma \ref{homotopyII} the second assertion of Theorem \ref{main}.


\subsection{Proof of Theorem \ref{theorem:bifurcation}}

In this final part of Section \ref{section:proofs} we prove Theorem \ref{theorem:bifurcation}. Throughout the proof, $I=[a,b]$ denotes a compact interval in $\mathbb{R}$. At first we introduce a bifurcation theorem for continuous paths of $C^2$ functionals of Fredholm type (cf. \cite[Thm. 2.1]{JacBifIch}).

\begin{theorem}\label{SFLPejsachowiczIIbif}
Let $H$ be a separable Hilbert space and $\psi:I\times H\rightarrow\mathbb{R}$ a continuous map such that each $\psi_\lambda=\psi(\lambda,\cdot):H\rightarrow\mathbb{R}$, $\lambda\in I$, is $C^2$ and its derivatives depend continuously on $(\lambda,u)\in I\times H$. Assume that $0$ is a critical point of each $\psi_\lambda$ and that the corresponding Hessians $L_\lambda$ are Fredholm with $L_a$ and $L_b$ invertible. If $\sfl(L)\neq 0$, then the interval $(a,b)$ contains at least one point of bifurcation of critical points of $\psi$ from the trivial branch $I\times\{0\}\subset I\times H$.   
\end{theorem}
Note that, if we apply Theorem \ref{SFLPejsachowiczIIbif} to the functionals $\psi_\lambda$ from Section \ref{section:indthm} in the case that $(X,Y)=(I,\partial I)$, we obtain bifurcation of weak solutions of the Hamiltonian systems \eqref{equation} from the branch $I\times\{0\}\subset I\times H^\frac{1}{2}(S^1,\mathbb{R}^{2n})$ provided that the spectral flow of the path $L$ in \eqref{Lformula2} is non-zero.\\  
For the proof of Theorem \ref{theorem:bifurcation} let now $X$ be a path-connected topological space and $\emptyset\neq Y\subset X$ closed such that $L_\lambda$ is invertible for all $\lambda\in Y$. By the implicit function theorem, bifurcation can occur only at points $\lambda\in X$ where $L_\lambda$ is non-invertible and consequently $Y$ is contained in $X\setminus B(\mathcal{H})$.\\
We fix a point $\lambda_0\in Y$ and let  $\tilde{\gamma}_1,\tilde{\gamma}_2$ and $\tilde{\gamma}_3$ be three paths such that $\tilde{\gamma}_1(a)=\tilde{\gamma}_3(a)=\lambda_0$, $\tilde{\gamma}_1(b)=\tilde{\gamma}_2(a)\in Y$ and $\tilde{\gamma}_3(b)=\tilde{\gamma}_2(b)\in Y$. This yields a path of operators $B:I\rightarrow\mathcal{FS}(H^\frac{1}{2}(S^1,\mathbb{R}^{2n}))$ by $B_\lambda=(L\circ(\tilde{\gamma}_1\ast\tilde{\gamma}_2\ast\tilde{\gamma}^{-1}_3))_\lambda$, where $\tilde{\gamma}^{-1}_{3}(t)=\tilde{\gamma}_3(a+b-t)$, $t\in I$, denotes the inverse path of $\tilde{\gamma}_3$. Since $B$ is closed, we conclude from Lemma \ref{comppert} and Lemma \ref{lemma:propertiessfl} iii) that $\sfl(B)=0$. From this and the first two assertions of Lemma \ref{lemma:propertiessfl}, it follows that

\begin{align}\label{Bifpaths}
\sfl(L\circ\tilde{\gamma}_2)=\sfl(L\circ\tilde{\gamma}_3)-\sfl(L\circ\tilde{\gamma}_1).
\end{align}
Let now $C$ be a path component of $X\setminus B(\mathcal{H})$ such that $C\cap Y\neq\emptyset$. We assign an integer to $C$ by $\iota(C)=\sfl(L\circ\tilde{\gamma}_1)$, where $\tilde{\gamma}_1$ is any path which connects $\lambda_0$ and some point $\lambda_1\in C\cap Y$. In order to show that $\iota(C)$ is well defined, let $\tilde{\gamma}_3$ be another path connecting $\lambda_0$ and some point $\lambda_2\in C\cap Y$. We join $\lambda_1$ and $\lambda_2$ by a path $\tilde{\gamma}_2$ in $C$ and consider the continuous path of $C^2$ functionals 

\[\tilde\psi:I\times H^\frac{1}{2}(S^1,\mathbb{R}^{2n})\rightarrow\mathbb{R},\quad \tilde\psi(\lambda,u)=\psi(\tilde{\gamma}_2(\lambda),u)\] 
The critical points of $\tilde\psi$ are the weak solutions of the Hamiltonian systems

\begin{equation}\label{hambifequ}
\left\{
\begin{aligned}
Ju'(t)+\nabla_u\mathcal{H}&(\tilde{\gamma}_2(\lambda),t,u(t))=0,\quad \lambda\in I,\\
u(0)&=u(2\pi),
\end{aligned}
\right.
\end{equation}
and the Hessians at $0\in H^\frac{1}{2}(S^1,\mathbb{R}^{2n})$ are given by the operators $L\circ\tilde{\gamma}_2$. We conclude from Theorem \ref{SFLPejsachowiczIIbif} that $\sfl(L\circ\tilde{\gamma}_2)=0$, because otherwise there is a bifurcation point of \eqref{hambifequ} and consequently a bifurcation point for \eqref{equation} on $\tilde{\gamma}_2(I)\subset C\subset X\setminus B(\mathcal{H})$. Therefore, we see by \eqref{Bifpaths} that $\iota(C)$ is well defined.\\
Now, given any path $\gamma:(I,\partial I)\rightarrow (X,Y)$, it follows from Theorem \ref{mainpath} and \eqref{Bifpaths} that

\[d(\gamma)=\sfl(L\circ\gamma)=\iota(C_{\gamma(b)})-\iota(C_{\gamma(a)}),\] 
where we denote by $C_\lambda$ the component of $X\setminus B(\mathcal{H})$ to which $\lambda\in Y$ belongs. This proves Theorem \ref{theorem:bifurcation}, since if $d(\gamma)\neq 0$ for some $\gamma:(I,\partial I)\rightarrow (X,Y)$, it may be concluded that $X\setminus B(\mathcal{H})$ has at least two path components. Moreover, if there exists a sequence $\gamma_n:(I,\partial I)\rightarrow (X,Y)$ such that $|d(\gamma_n)|\rightarrow\infty$, it follows that that the number of path components of $X\setminus B(\mathcal{H})$ cannot be finite.


\appendix

\section{K-theory}
The aim of this appendix is to recall briefly the definition and main properties of topological $K$-theory of locally compact spaces. Our main references are \cite{Segal} and \cite{Lawson}.\\
Let $X$ be a locally compact topological space. We consider triples $\{E_0,E_1,a\}$, where $E_0$ and $E_1$ are complex vector bundles over $X$ and $a:E_0\rightarrow E_1$ is a bundle morphism. The \textit{support} $\supp\xi$ of such a triple $\xi=\{E_0,E_1,a\}$ is defined to be the closed subset of $X$ consisting of those points $\lambda\in X$ for which $a_\lambda:E_{0,\lambda}\rightarrow E_{1,\lambda}$ is not an isomorphism. $\xi$ is said to be \textit{trivial} if $\supp\xi=\emptyset$. We call $\{E^0_0,E^0_1,a_0\}$ and $\{E^1_0,E^1_1,a_1\}$ \textit{isomorphic}, if there exist bundle isomorphisms $\varphi_0:E^0_0\rightarrow E^1_0$ and $\varphi_1:E^0_1\rightarrow E^1_1$ such that the diagram

\begin{align*}
\xymatrix{E^1_0\ar[r]^{a_1}&E^1_1\\
E^0_0\ar[u]^{\varphi_0}\ar[r]^{a_0}&E^0_1\ar[u]_{\varphi_1}
}
\end{align*}  
commutes.\\
For a closed subspace $Y\subset X$, we denote by $L(X,Y)$ the set of isomorphism classes of all triples $\xi=\{E_0,E_1,a\}$ on $X$ such that $\supp\xi$ is a compact subset of $X\setminus Y$. Note that $L(X,Y)$ is a semigroup under the operation of direct sum and

\[\supp \xi^0\oplus \xi^1=\supp \xi^0\cup\supp \xi^1,\quad \xi^0,\xi^1\in L(X,Y).\]
We call $\xi^0=\{E^0_0,E^0_1,a_0\}$ and $\xi^1=\{E^1_0,E^1_1,a_1\}$ \textit{homotopic}, $\xi^0\simeq \xi^1$, if there is an element in $L(X\times[0,1],Y\times[0,1])$ such that its restrictions to $X\times\{0\}$ and $X\times\{1\}$ are isomorphic to $\xi^0$ and $\xi^1$, respectively. Finally, we introduce an equivalence relation $\sim$ on $L(X,Y)$ by $\xi^0\sim \xi^1$ if there are trivial elements $\eta^0,\eta^1\in L(X,Y)$ such that 

\[\xi^0\oplus \eta^0\simeq \xi^1\oplus \eta^1.\]
The \textit{K-theory} $K(X,Y)$ of the pair $(X,Y)$ is defined as the set of equivalence classes of $L(X,Y)$ with respect to this equivalence relation. Henceforth we denote the class of $\{E_0,E_1,a\}\in L(X,Y)$ in $K(X,Y)$ by $[E_0,E_1,a]$. The sum operation

\[[E^0_0,E^0_1,a_0]+[E^1_0,E^1_1,a_1]=[E^0_0\oplus E^1_0, E^0_1\oplus E^1_1, a_0\oplus a_1]\in K(X,Y)\]
turns $K(X,Y)$ into an abelian group, where the neutral element is represented by any trivial element in $L(X,Y)$.\\
For a proper map $f:(X,Y)\rightarrow (X',Y')$ of topological pairs, we obtain a homomorphism

\[f^\ast:K(X',Y')\rightarrow K(X,Y),\quad f^\ast[E_0,E_1,a]=[f^\ast E_0,f^\ast E_1,f^\ast a],\]
where $f^\ast E_0$, $f^\ast E_1$, denote the pullback bundles and $(f^\ast a)_\lambda=a_{f(\lambda)}$, $\lambda\in X$. If $g:(X,Y)\rightarrow (X',Y')$ and $f:(X',Y')\rightarrow (\tilde{X},\tilde{Y})$ are proper, then

\[(f\circ g)^\ast=g^\ast\circ f^\ast:K(\tilde{X},\tilde{Y})\rightarrow K(X,Y).\]
Moreover, $f^\ast=g^\ast:K(X',Y')\rightarrow K(X,Y)$ if $f\simeq g:(X,Y)\rightarrow(X',Y')$ are properly homotopic maps. Consequently, $K$ is a contravariant functor from the category of pairs of locally compact spaces and closed subspaces to the category of abelian groups. A slightly different homotopy invariance property, which is often useful in computations, reads as follows:

\begin{lemma}\label{homotopyII}
Let $E_0$ and $E_1$ be vector bundles over $X$ and $a:[0,1]\rightarrow \hom(E_0,E_1)$, a path of bundle morphisms. Set $\xi_t:=\{E_0,E_1,a_t\}$ and assume that $\supp\xi_t\subset K\subset X\setminus Y$, $t\in [0,1]$, for some compact set $K\subset X$. Then

\[[E_0,E_1,a_0]=[E_0,E_1,a_1]\in K(X,Y).\] 

\end{lemma}

The following lemma states the so called \textit{logarithmic property} of $K$.

\begin{lemma}\label{logarithmic}
If $[E_0,E_1,a_0]$, $[E_1,E_2,a_1]\in K(X,Y)$, then their sum is given by

\[[E_0,E_1,a_0]+[E_1,E_2,a_1]=[E_0,E_2,a_1\circ a_0]\in K(X,Y).\]
\end{lemma}

The \textit{odd $K$-theory groups} are defined by

\[K^{-1}(X,Y)=K(X\times\mathbb{R},Y\times\mathbb{R})\]
and, as above, any proper map $f:(X,Y)\rightarrow(X',Y')$ induces a homomorphism

\[f^\ast:K^{-1}(X',Y')\rightarrow K^{-1}(X,Y).\]
Finally, we want to recall the definition of the well known isomorphism $c_1:K^{-1}(I,\partial I)\rightarrow\mathbb{Z}$ coming from the first Chern number, where $I\subset\mathbb{R}$ is any compact interval. Let 

\[[E_0,E_1,a]\in K^{-1}(I,\partial I)=K(I\times\mathbb{R},\partial I\times\mathbb{R})\]
be given. Since $I\times\mathbb{R}$ is contractible, we can find global trivialisations $\psi: E_0\rightarrow(I\times\mathbb{R})\times\mathbb{C}^n$ and $\varphi:E_1\rightarrow(I\times\mathbb{R})\times\mathbb{C}^n$. Then

\begin{align}\label{Chern}
c_1([E_0,E_1,a])=w(\det(\varphi\circ a\circ\psi^{-1})\circ\kappa,0)\in\mathbb{Z},
\end{align}
where $\kappa:S^1\rightarrow I\times\mathbb{R}$ is any simple positively oriented curve surrounding the support of $\{E_0,E_1,a\}$, and $w(\cdot,0)$ denotes the winding number for closed curves in $\mathbb{C}\setminus\{0\}$ with respect to $0$.


\section{Spectral Flow}
The spectral flow of paths of selfadjoint Fredholm operators was introduced by Atiyah, Patodi and Singer in \cite{AtiyahPatodi} in connection with the eta invariant and spectral asymmetry. Since then it has reappeared in several other contexts like Floer homology, the distribution of eigenvalues of the Dirac operator, odd Chern characters, gauge anomalies and bifurcation theory, among others. Detailed references can be found in the introduction of \cite{SFLPejsachowicz}. We base our presentation essentially on the recent work \cite{UnbSpecFlow}.\\
We follow the notation of Section \ref{section:sindbundle} and assume that $W$ and $H$ are complex Hilbert spaces with a dense continuous inclusion $W\hookrightarrow H$. Let us recall that for a selfadjoint Fredholm operator $T_0\in\mathcal{FS}(W,H)$, $0\in\mathbb{R}$ either belongs to the resolvent set or it is an isolated eigenvalue of finite multiplicity. Since the set of Fredholm operators is open in $\mathcal{L}(W,H)$, there exists $\Lambda>0$ such that $\pm\Lambda$ do not belong to the spectrum 

\[\sigma(T_0)=\{\lambda\in\mathbb{R}:\,\lambda-T_0\notin\mathcal{GS}(W,H)\}\] 
of $T_0$ and $\sigma(T_0)\cap[-\Lambda,\Lambda]$ consists only of isolated eigenvalues of finite multiplicity. If now $\Gamma$ denotes the simple closed curve surrounding $0$ positively along the circle of radius $\Lambda$, then

\[\chi_{[-\Lambda,\Lambda]}(T_0)=\frac{1}{2\pi i}\int_\Gamma{(\lambda-T_0)^{-1}\,d\lambda},\]
is the orthogonal projection onto the direct sum of the eigenspaces of $T_0$ with respect to eigenvalues in $[-\Lambda,\Lambda]$. In what follows we set for $-\Lambda\leq c<d\leq\Lambda$

\[E_{[c,d]}(T_0)=\bigoplus_{\lambda\in[c,d]}\ker(\lambda-T_0).\]
From the continuity of spectra, it is readily seen that there exists a neighbourhood $N(T_0,\Lambda)\subset\mathcal{FS}(W,H)$ of $T_0$ such that $\pm\Lambda\notin\sigma(T)$ and the rank of the projection $\chi_{[-\Lambda,\Lambda]}(T)$ has the same finite value for all $T\in N(T_0,\Lambda)$.\\
Let now $\mathcal{A}:[a,b]\rightarrow\mathcal{FS}(W,H)$ be a path of selfadjoint Fredholm operators. We choose a subdivision $a=t_0<t_1<\ldots<t_N=b$, operators $T_i\in\mathcal{FS}(W,H)$ and numbers $\Lambda_i>0$, $i=1,\ldots N$, such that the restriction of the path $\mathcal{A}$ to $[t_{i-1},t_i]$ runs entirely inside $N(T_i,\Lambda_i)$. Accordingly, $\dim E_{[-\Lambda_i,\Lambda_i]}(\mathcal{A}_t)$ is constant for $t\in [t_{i-1},t_i]$, $i=1,\ldots,N$. The \textit{spectral flow} of $\mathcal{A}$ is defined by

\begin{align}\label{sfl}
\sfl(\mathcal{A})=\sum^N_{i=1}{\dim E_{[0,\Lambda_i]}(\mathcal{A}_{t_i})-\dim E_{[0,\Lambda_i]}(\mathcal{A}_{t_{i-1}})}\in\mathbb{Z}.
\end{align}  
Note that, roughly speaking, $\sfl(\mathcal{A})$ counts the number of negative eigenvalues of $\mathcal{A}_a$ that
become positive as the parameter $t$ travels from $a$ to $b$ minus the number of positive eigenvalues of $\mathcal{A}_a$ that become negative; i.e. the net number of eigenvalues which cross zero\\
There exist other definitions of the spectral flow under various assumptions in the literature. We want to mention here the work of Robbin, Salamon \cite{Robbin-Salamon} in the case $W\neq H$ and Fitzpatrick, Pejsachowicz, Recht \cite{SFLPejsachowicz} for $W=H$. The relation between those several existing definitions, which all have the same interpretation, was clarified by Lesch in \cite{LeschSpecFlowUniqu}. Let us denote by $\Omega(\mathcal{FS}(W,H),\mathcal{GS}(W,H))$ the set of all paths in $\mathcal{FS}(W,H)$ having ends in $\mathcal{GS}(W,H)$. By using Proposition 4.6 of \cite{Robbin-Salamon}, one may restate the uniqueness of the spectral flow as proved in \cite[Thm. 5.13]{LeschSpecFlowUniqu} as follows:

\begin{theorem}\label{Lesch}
Let

\begin{align*}
\mu:\Omega(\mathcal{FS}(W,H),\mathcal{GS}(W,H))\rightarrow\mathbb{Z}
\end{align*}
be a map such that

\begin{enumerate}
	\item[i)] $\mu(\mathcal{A}^1)=\mu(\mathcal{A}^2)$ if $\mathcal{A}^1,\mathcal{A}^2$ are homotopic inside $\Omega(\mathcal{FS}(W,H),\mathcal{GS}(W,H))$.
	\item[ii)] If $\mathcal{A}^1,\mathcal{A}^2\in\Omega(\mathcal{FS}(W,H),\mathcal{GS}(W,H))$, then
	\begin{align*}
	\mu(\mathcal{A}^1\oplus\mathcal{A}^2)=\mu(\mathcal{A}^1)+\mu(\mathcal{A}^2).
	\end{align*}
	\item[iii)] If $\mathcal{A}\in\Omega(\mathcal{FS}(W,H),\mathcal{GS}(W,H))$ is constant, then $\mu(\mathcal{A})=0$.
	\item[iv)] There is a rank one orthogonal projection $P$ with $\ker(P)\subset W$ such that for all $T\in\mathcal{S}(W,H)$ with $(I_H-P)T(I_H-P)\in\mathcal{GS}(\ker(P))$ the path 
	
	\[	\mathcal{A}_t=t\, P+(I_H-P)T(I_H-P),\quad t\in[-1,1],\]
	has $\mu(\mathcal{A})=1$.
\end{enumerate}
Then $\mu$ equals the spectral flow.
\end{theorem}      

Besides the four characteristics stated in Theorem \ref{Lesch}, we want to mention three further properties in the following lemma.

\begin{lemma}\label{lemma:propertiessfl}
Let $\mathcal{A},\mathcal{A}^1,\mathcal{A}^2:[a,b]\rightarrow\mathcal{FS}(W,H)$ be three paths of selfadjoint Fredholm operators.
\begin{itemize}
	\item[i)] $\sfl(\mathcal{A}^1\ast\mathcal{A}^2)=\sfl(\mathcal{A}^1)+\sfl(\mathcal{A}^2)$, whenever the concatenation of the paths $\mathcal{A}^i$, $i=1,2$, is defined,
	\item[ii)]  $\sfl(\mathcal{A}')=-\sfl(\mathcal{A})$, where $\mathcal{A}'(t)=\mathcal{A}(a+b-t)$ denotes the inverse path of $\mathcal{A}$,
	\item[iii)] if there exists a path of selfadjoint operators $K:[a,b]\rightarrow\mathcal{S}(H)$ such that $\mathcal{A}^1_{t}-\mathcal{A}^2_{t}=K_t\mid_W\in\mathcal{K}(W,H)$, $t\in[a,b]$, and $K_a=K_b=0$, then
	
	\[\sfl(\mathcal{A}^1)=\sfl(\mathcal{A}^2).\]
\end{itemize}
\end{lemma}
Note that i) and ii) follow immediately from \eqref{sfl}, whereas iii) is a simple consequence of the homotopy invariance stated in Theorem \ref{Lesch}.\\
Finally, we want to point out that one can define the spectral flow for paths of selfadjoint Fredholm operators acting on real Hilbert spaces verbatim by \eqref{sfl}. On the other hand, if $W$ and $H$ are real Hilbert spaces and $T\in\mathcal{FS}(W,H)$ is a selfadjoint Fredholm operator, then the complexification $T^\mathbb{C}$ acting between the complexified spaces $W^\mathbb{C}$ and $H^\mathbb{C}$ defines an element of $\mathcal{FS}(W^\mathbb{C},H^\mathbb{C})$ (cf. \cite{Weidmann}). Moreover, $\sigma(T)=\sigma(T^\mathbb{C})$ and $\dim E_{[a,b]}(T)=\dim E_{[a,b]}(T^\mathbb{C})$ for every compact interval $[a,b]$ such that $[a,b]\cap\sigma(T)$ consists only of isolated eigenvalues of finite multiplicity and $a,b\notin\sigma(T)$. Now given a path $\mathcal{A}:[a,b]\rightarrow\mathcal{FS}(W,H)$, it follows from the definition in \eqref{sfl} that 

\begin{align}\label{invsfl}
\sfl(\mathcal{A})=\sfl(\mathcal{A}^\mathbb{C}),
\end{align}
where $\mathcal{A}^\mathbb{C}:[a,b]\rightarrow \mathcal{FS}(W^\mathbb{C},H^\mathbb{C})$ denotes the path of complexified operators.


\subsubsection*{Acknowledgements}
This paper was written while the author enjoyed the kind hospitality of the Dipartimento di Scienze Matematiche "Giuseppe Luigi Lagrange" at the Politecnico di Torino in Italy. Special thanks go to "the girls": Teresa Fischer and Bruna Gianotti. Mi mancherete!

\vspace{1cm}
Nils Waterstraat\\
Institut für Mathematik\\
Humboldt-Universität zu Berlin\\
Unter den Linden 6\\
10099 Berlin\\
Germany\\
E-mail: waterstn@math.hu-berlin.de

\end{document}